\documentclass[10pt, oneside]{article}   	
\usepackage{geometry}                		
\usepackage{graphicx}				
\usepackage{amssymb}
\usepackage{amsthm}
\usepackage{graphicx}
\usepackage{mathtools}
\usepackage{amsmath,amsfonts}
\usepackage{epsfig}
\usepackage[usenames, dvipsnames]{color}

\newcommand{\Z}{{\mathbb Z}}
\newcommand{\R}{{\mathbb R}}

\newcommand{\N}{{\mathbb N}}

\newcommand{\Om}{\Omega}

\newcommand{\pip}{\varphi}
\newcommand{\eps}{\varepsilon}
\newcommand{\dmu}{d\mu}

\newcommand{\loc}{{\mbox{\scriptsize{loc}}}}

\newcommand{\bdry}{\partial}

\newcommand{\8}{\infty}

\DeclareMathOperator{\diam}{diam}
\DeclareMathOperator{\Lip}{Lip}
\DeclareMathOperator{\lip}{lip}
\DeclareMathOperator{\LIP}{LIP}
\DeclareMathOperator{\var}{osc}

\DeclareMathOperator{\supp}{supp}
\DeclareMathOperator{\rad}{\textrm{rad}}

\def\vint_#1{\mathchoice
	{\mathop{\vrule width 5pt height 3 pt depth -2.5pt
			\kern -9pt \kern 1pt\intop}\nolimits_{\kern -5pt{#1}}}
	{\mathop{\vrule width 5pt height 3 pt depth -2.6pt
			\kern -6pt \intop}\nolimits_{\kern -3pt{#1}}}
	{\mathop{\vrule width 5pt height 3 pt depth -2.6pt
			\kern -6pt \intop}\nolimits_{\kern -3pt{#1}}}
	{\mathop{\vrule width 5pt height 3 pt depth -2.6pt
			\kern -6pt \intop}\nolimits_{\kern -3pt{#1}}}}

\theoremstyle{plain}
\newtheorem{theorem}{Theorem}[section]

\newtheorem{lemma}[theorem]{Lemma}
\newtheorem{proposition}[theorem]{Proposition}
\newtheorem{theorema}{Theorem}

\theoremstyle{definition}
\newtheorem{definition}[theorem]{Definition}

\newtheorem{remark}[theorem]{Remark}

\newtheorem{Example}[theorem]{Example}

\numberwithin{equation}{section}

\title{Asymptotic behavior of BV functions and sets of \\
finite perimeter in metric measure spaces}
\author{Sylvester Eriksson-Bique, James T. Gill,\\ Panu Lahti, and Nageswari Shanmugalingam
\footnote{The first author was supported by NSF grant \#DMS-1704215.
The third author was partially supported by the Finnish Cultural Foundation.
The fourth author was partially supported by the NSF grant \#DMS-1500440 (U.S.).
The motivation for this
research came from a recommendation by Tatiana Toro to study asymptotic behavior of sets of finite perimeter, during
the fourth author's visit to University of Washington at Seattle in 2010; the authors wish to thank her for the valuable 
comments that motivated this research. Part of the research was done during the visit
of the second and third authors to University of Cincinnati, and during the period the first, third, and fourth authors were visiting
Link\"oping University supported by a departmental grant from the 
Wallenberg Foundation; they wish to thank those institutions for their kind hospitality. We also wish to thank
Tapio Rajala for valuable discussions that lead to Example~\ref{ex:nonmin}.}}

\begin{document}
\maketitle

\noindent{\small
{\bf Abstract:} In this paper, we study the asymptotic behavior of BV functions in complete 
metric measure
spaces equipped with a doubling measure supporting a $1$-Poincar\'e inequality. We show that 
at almost every point $x$ outside the Cantor and jump parts of a BV function, the asymptotic limit
of the function is a Lipschitz continuous function of least gradient on a tangent space to the metric
space based at $x$. We also show that, at co-dimension $1$ Hausdorff measure
almost every measure-theoretic boundary point of a set $E$ of finite perimeter, there is an asymptotic
limit set $(E)_\8$ corresponding to the asymptotic expansion of $E$  and that 
every such asymptotic limit $(E)_\8$ is
a quasiminimal set of finite perimeter. 
We also show that the perimeter measure of $(E)_\8$ is Ahlfors co-dimension $1$ regular.
}

\bigskip


\medskip


\section{Introduction}

The classical notion of differentiability  for a  function $f$ on  a subset of Euclidean space 
$\Om\subset\R^n$  at a point 
$x\in\Om$ is that the graph of $f$ should, at $(x,f(x))\in\R^n\times\R$, 
asymptotically 
 approach an $n$-dimensional hyperplane in
$\R^{n+1}$. 
In other words, the function $f$ behaves asymptotically like an affine function. This notion has been extended to mappings between domains
in Riemannian manifolds in the study of differential geometry. 

The seminal work of Cheeger~\cite{Che} 
extended the above notion of affine approximation to 
the realm of metric measure spaces,
with \emph{generalized linear functions} defined on measured Gromov-Hausdorff tangent 
spaces playing the role of affine functions.
It was shown there that, if the space is complete, the measure is doubling, and 
the space supports a $p$-Poincar\'e inequality for some $1\le p<\infty$, then
 every Lipschitz function $f$ on the metric space is asymptotically generalized linear
 at almost every point $x$ in the space. More specifically, we have the following: let $X_\infty$ be obtained as a pointed
measured Gromov-Hausdorff limit of scaled versions $(X_n,d_n,x,\mu_n)$ of the metric measure space
$(X,d,\mu)$ with $x\in X$. 
In considering corresponding scaled 
versions $f_n\colon X_n\to\R$ of $f\colon X\to\R$, where
\[
f_n(y):=\frac{f(y)-f(x)}{r_n}
\]
with $\{r_n\}_{n\in \N}$
a sequence of positive numbers decreasing to zero 
which form the scales associated with the metric
$d_n:=r_n^{-1}d$ in the Gromov-Hausdorff limit,
the sequence of functions $f_n$ converges to a limit function $f_\infty\colon X_\infty\to\R$ (after passing to a subsequence if necessary).
Cheeger proved that this asymptotic limit function $f_\infty$ is
a generalized linear function on $X_\infty$.  Here by
a generalized linear function, the paper~\cite{Che} means a function that is $p$-harmonic on 
$X_\infty$ with a constant function
as its minimal $p$-weak upper gradient.


In this paper, we extend the study of asymptotic behavior of Lipschitz functions in~\cite{Che} to 
functions of bounded variation in complete metric measure spaces equipped with a doubling
measure supporting a $1$-Poincar\'e inequality. 
The following theorems give a summary of the principal results of this paper; the precise versions
can be found in the statements of the corresponding theorems in Sections~4--6. 

\begin{theorema}[Theorems \ref{thm:main1} and \ref{thm:tangentharm}] 
Let $u$ be a function of bounded variation on $X$.  For $\mu$-a.e. $x \in X$ and any 
tangent space $(X_\infty, d_\infty, x_\infty, \mu_\infty)$  of $X$ based at $x$,
any limit function $u_\infty$  as described above is $1$-harmonic (also known as
function of least gradient) and has quasi-constant minimal $1$-weak upper gradient. 
\end{theorema}


The most fundamental of BV functions are characteristic functions of sets of finite perimeter.
For these functions, the most interesting behavior happens solely at their jump points. Here the 
study of asymptotics is different, see e.g.~\cite[Theorem~5.13]{EG} in the Euclidean setting. 
Similarly, for general BV functions $u$, the approach of scaling the function as described above works
well when considering points in $X$ that asymptotically see neither the Cantor nor the jump parts of the
variation measure $\Vert Du\Vert$ of $u$, but it is not helpful in the study of asymptotic behavior of $u$ at
points in its jump set $S_u$.  
Instead, an approach based on weak* limits of measures, which can also be used to define the limit function $u_\infty$ as in 
Theorem~\ref{thm:main3}, is more in line with studying the
behavior of $u$ at points in the jump set $S_u$ of $u$ and gives an alternative approach to Theorem~A.
This measure-theoretic approach 
is applied to characteristic functions of sets of finite perimeter in Sections~5 and~6  and the main conclusions are 
described in Theorems~B and C below. 

\begin{theorema}[Theorem \ref{thm:main3}] 
Let $E \subset X$ be of finite perimeter $P(E, \cdot)$. Then for 
$P(E,\cdot)$-a.e. point, appropriately scaled versions of $P(E, \cdot)$ converge to a measure on 
$X_\infty$ that is comparable to the co-dimension $1$ Hausdorff measure.
\end{theorema}

In $\R^n$, the corresponding limits are not just $(n-1)$-dimensional, but are hyperplanes
that are boundaries of sets whose characteristic functions are functions of least gradient, that is,
of minimal boundary surface. In the metric setting, we obtain an analogue with quasiminimal 
sets playing the role of hyperplanes and minimal boundary surfaces.


\begin{theorema}[Theorem \ref{thm:minimaltangent}] 
Let $E\subset X$ be of finite perimeter. Then, with respect to the co-dimension $1$ Hausdorff measure,
almost every point $x$ on the measure-theoretic boundary of $E$ satisfies the following properties: fixing a (pointed)
tangent space $(X_\8, d_\8,x_\8,\mu_\8)$ arising as a Gromov-Hausdorff limit of the scaled sequence
$(X_n,d_n,x,\mu_n)$, and by passing to a subsequence if necessary we obtain that
\begin{itemize}
\item the sequence of measures $\chi_E\, d\mu_n$ on $X_n$ converges weakly* to a measure $\mu_\8^E$ on $X_\8$,
\item this limit measure is absolutely continuous with respect to $\mu_\8$, 
\item there is a set $(E)_\8\subset X_\8$ of finite perimeter such that $d\mu_\8^E=\chi_{(E)_\8}\, d\mu_\8$,
\item the set $(E)_\8$ is of quasiminimal boundary surface {\rm(}see Definition~\ref{def:Q-min} or~\cite{KKLS}{\rm)}, and
\item the measure described in Theorem~B is supported on the boundary of $(E)_\8$, and is comparable to the
perimeter measure $P((E)_\8,\cdot)$ of $(E)_\8$.
\end{itemize}
\end{theorema} 

Thus, beginning with extensions of the Cheeger's Rademacher theorem for doubling metric spaces with a Poincar\'e inequality to the functions of bounded variation, we recover important aspects of the classical theory of the boundaries of finite perimeter sets in $\R^n$.


The class of BV functions considered here is based on the notion first proposed by Miranda Jr.~\cite{Mir}, and was further developed 
in~\cite{A, AMP,AmbDiM}. 
The corresponding notion of a function of least gradient was studied in~\cite{KKLS, HKLS, KLLS}. Just as~\cite{Che} related
asymptotic limits of Lipschitz functions to generalized linear functions (which are a priori $p$-harmonic for the indices
$p>1$ for which $X$ supports a $p$-Poincar\'e inequality), we relate asymptotic limits of BV functions to functions of
least gradient when the point of asymptoticity does not lie in the set where the jump and Cantor parts of the variation measure
live. Additionally,  at 
almost every point  with respect to the co-dimension $1$ Hausdorff measure in the measure-theoretic boundary of the set of finite 
perimeter, we relate the asymptotic limit of that set
to sets of finite perimeter that have a quasiminimal boundary surface as in~\cite{KKLS}.

In the setting of Heisenberg groups (perhaps the simplest non-Riemannian example of the type of metric measure spaces
studied here), more is known of the asymptotic behavior of BV functions; the key papers to study this setting are those of
Magnani~\cite{Mag},
Franchi, Serapioni and Serra-Cassano~\cite{FSS}, and Ambrosio, Ghezzi and Magnani~\cite{AGM}. 
It is shown in~\cite[Theorem~4.1]{FSS} that asymptotic limits
of sets of finite perimeter in a Heisenberg group, based at a reduced boundary point of that set, is a Euclidean
(vertical) half-space with the boundary plane parallel to the non-horizontal direction. Studies of asymptotic limits of sets
of finite perimeter in more general step-$2$ Carnot groups can be found in~\cite{FSS2}, and for more general
Carnot groups in~\cite{FSS3}. While the Heisenberg groups are topologically Euclidean, there are more sets of
finite perimeter in the Heisenberg sense than in the Euclidean sense, see~\cite[Proposition~2.15]{FSS}. The 
papers~\cite{FSS, FSS2, FSS3} rely on the group structure on the Carnot groups, and so they do not address
the case of more general Carnot-Carath\'eodory spaces. 

Carnot-Carath\'eodory spaces are (locally)
doubling metric measure spaces supporting a $1$-Poincar\'e inequality, and hence the results of the present paper 
also apply there. Note that tangent spaces of Carnot-Carath\'eodory spaces are topological groups equipped with
dilation operations, and if the tangent space is based at a regular point of the Carnot-Carath\'eodory space, then
it is a nilpotent group equipped with a dilation,  see~\cite{Mit, Bel, LeDo}. Under further assumptions on the 
Carnot-Carath\'eodory space (which lead to knowing that the tangent spaces are all Carnot groups), 
a similar asymptoticity study is undertaken in~\cite{AGM}. We point out here that the results in the current paper
are applicable to all Carnot-Carath\'eodory spaces of topological dimension at least $2$.

 If $\nu$ is a Radon measure on $X$ and $x\in X$, then for almost every $r>0$ we know that
$\nu(\overline{B}(x,r)\setminus B(x,r))=0$. If $X$ is a geodesic space and $\mu$ is a doubling measure, then
$\mu(\overline{B}(x,r)\setminus B(x,r))=0$ for \emph{each} $r>0$ and $x\in X$, see~\cite{Buc}. In this paper,
we will assume that $X$ is geodesic in order to simplify many of the proofs (by avoiding the discussion of 
having to slightly adjust the radius $r$ in order to ensure that $\mu(\overline{B}(x,r)\setminus B(x,r))=0$), but
our results hold also in spaces that are \emph{not} geodesic by an easy (but
notationally cumbersome) modification discussed in Section 2 below.

 The structure of this paper is as follows. In Section 2 we give the basic definitions necessary 
for the study of sets of finite perimeter and functions
of bounded variation on metric measure spaces.  In Section 3 we discuss pointed measured Gromov-Hausdorff 
limits.  In Section 4 we show the 
results stated above regarding that asymptotic limits of BV functions converge to a function of least gradient ($1$-harmonic)
in the tangent space, see Theorem~\ref{thm:tangentharm}.  
In Section 5 we discuss asymptotic limits of a set of finite perimeter, and show that
for co-dimension $1$ almost every point on the measure-theoretic boundary of that set we have a 
tangential behavior of the set; more specifically, there is a Gromov-Hausdorff type limit $(E)_\8$ of the set $E$ at 
such a point, and this limit is a set of (locally) finite perimeter; this is the content of
Theorem~\ref{thm:main3}. We also verify certain geometric structural
regularity of these limit sets, see Theorem~\ref{thm:Main2}. The final section of this paper is devoted to the discussion
on asymptotic minimality for sets of finite perimeter. In Theorem~\ref{thm:minimaltangent} we show that these limit sets
$(E)_\8$ are sets of quasiminimal boundary surfaces.
%

\section{Notation and definitions}

Here we lay out the main definitions and assumptions for this paper.  Much of the terminology will be similar to that 
used in~\cite{A, AMP, Mir}.

We assume that $(X,d,\mu)$ is a complete metric measure space
with $\diam X>0$, that is, $X$ consists of at least two points.
We use the notation $B(x,r)$ for the open ball 
centered at $x \in X$ and of radius $r>0$.  If we wish to be specific that the ball is in
the metric space $X$, we write $B_X(x,r)$.
Given a ball $B=B(x,r)$, we sometimes denote $\rad B:=r$;
note that in metric spaces, a ball (as a set) does not necessarily
have a unique center and radius, but we understand these to be
prescribed for all balls that we consider.
We will always assume that $\mu$ is {\em doubling}: there is a constant
$C_d \geq 1$ such that for all $x \in X$ and $r>0$,
\[ 
0<\mu(B(x,2r)) \leq C_d \mu(B(x,r))<\infty.
\]
By iterating the doubling condition, we obtain
for any $x\in X$ and any $y\in B(x,R)$ with $0<r\le R<\infty$ that
\begin{equation}\label{eq:homogenous dimension}
\frac{\mu(B(y,r))}{\mu(B(x,R))}\ge \frac{1}{C_d^2}\left(\frac{r}{R}\right)^{Q},
\end{equation}
where $Q>1$ only depends on the doubling constant $C_d$.

When a property holds outside a set of $\mu$-measure zero, we say that it holds
for $\mu$-a.e. $x\in X$.
As complete doubling metric 
spaces are proper, every closed and bounded set is compact.
Given an open set $W\subset X$, we take $\Lip_{\loc}(W)$ to be the space of functions 
on $W$ that are Lipschitz on every closed and bounded subset of $W$,
and $L^1_{\loc}(W)$ to be the space of functions 
integrable with respect to $\mu$ on every closed and bounded subset of $W$.

Given a rectifiable curve $\gamma \colon [0,1] \to X$, we define the length of $\gamma$ to be
\[ 
\ell(\gamma) := \sup \sum_i d(\gamma(t_i), \gamma(t_{i+1}))
\]
where the supremum is taken over all finite partitions $\{t_i\}$ of $[0,1]$.  We will always assume that $X$ 
is a {\em geodesic space}: for all $x, y$ in $X$,
\[ 
d(x,y) = \min \ell(\gamma) 
\]
where the minimum is taken over all curves $\gamma$ joining $x$ to $y$ and is achieved.

Given a function $u\colon X \to \mathbb{R}$,
an {\em upper gradient} $g$ of $u$ is a nonnegative 
Borel function such that for every $x,y \in X$ and every
rectifiable curve $\gamma$
containing $x$ and $y$, we have the inequality
\begin{equation}\label{eq:upper-grad}
|u(x)-u(y)| \leq \int_\gamma g \,ds,
\end{equation}
where $ds$ is arc length (see \cite{HKST} for more information and standard results about upper gradients).

We say that a family of rectifiable curves $\Gamma$ is of zero $p$-modulus,
for $1\le p<\infty$,
if there is a 
nonnegative Borel function $\rho\in L^p(X)$ such that 
for all curves $\gamma\in\Gamma$, the curve integral $\int_\gamma \rho\,ds$ is infinite.
If $g$ is a nonnegative $\mu$-measurable function on $X$
and (\ref{eq:upper-grad}) holds for all curves apart from
a family with zero $p$-modulus,
we say that $g$ is a $p$-weak upper gradient of $u$.
It is known that if a function $u$ on $X$ has an upper gradient in $L^p(X)$,
then there exists a minimal $p$-weak
upper gradient of $u$, denoted by $g_{u}$, satisfying $g_{u}\le g$ 
a.e. for any $p$-weak upper gradient $g\in L^{p}(X)$
of $u$, see \cite[Theorem 2.25]{BB}.

We 
always assume that the space $X$ supports a {\em $1$-Poincar\'e inequality}.
We say that $X$ supports a $p$-Poincar\'e inequality, for $1\le p<\infty$, if
there is a constant 
$C_P > 0$ so that for every $u \in \Lip_\loc (X)$, every upper gradient $g$ of $u$,
and every ball $B=B(x,r)$,
\[ \vint_B |u - u_B|\, d\mu \leq C_P r \left(\,\vint_{ B} g^p\, \dmu\right)^{1/p}, \]
where 
\[
u_{B}:=\vint_{B}u\,d\mu :=\frac 1{\mu(B)}\int_{B}u\,d\mu.
\]
We  will sometimes suppress the ``$1$-'' when discussing the inequality.  
We will denote by $C\ge 1$ a generic constant that only depends on
the doubling and Poincar\'e constants $C_d,C_P$, and whose
precise value may change even in the same line.


It can be noted that we could
only assume that
$X$ supports a \emph{weak} $1$-Poincar\'e inequality
(involving a ball dilated by a constant factor $\lambda> 1$ on the right-hand side),
and drop the assumption that the space is geodesic.  The reason
 for this is as follows:  a weak $1$-Poincar\'e inequality implies that the space is quasiconvex, and then a 
 bi-Lipschitz change in the metric will allow the space to become geodesic.  In geodesic spaces, a 
 weak Poincar\'e inequality can be improved to become strong.  This is discussed in~\cite{H}.  
 The class of functions of bounded variation,
 to be defined shortly, is invariant under a bi-Lipschitz 
 metric change.  Thus the assumptions of geodecicity and the strong version of the Poincar\'e inequality 
 are not restrictions, only conveniences. This bi-Lipschitz change in the metric on $X$ would induce
 a bi-Lipschitz change in the tangent space $X_\8$, with a bi-Lipschitz equivalent geodesic limit metric on $X_\8$ obtained as a limit of re-scaled geodesics
 metrics on $X$. We obtain that the asymptotic limit function $u_\8$ 
 as in Theorem~\ref{thm:tangentharm} is of least gradient with respect to this length metric on $X_\8$,
 and therefore is of quasi-least gradient with respect to the original metric on the tangent space $X_\8$.

We now wish to discuss functions of bounded variation and sets of finite
perimeter in the metric space $(X,d,\mu)$.  The 
definitions are quite different than those typically used for $X=\mathbb{R}^n$;
see \cite{Mir} for discussion
relating these to the classical definitions.  Many results from \cite{Mir} and \cite{A} will be used (and cited) 
in what follows.  For $u \in \Lip_\loc(X)$, we define
\[ \lip u(x) := \liminf_{r \to 0} \frac{\sup_{y \in B(x,r)} |u(y)-u(x)|}{r}, \]
often known as the {\em lower Lipschitz constant} of $u$ at $x$.
It is well known that $\lip u$ is an upper 
gradient of $u$ (see \cite[Section 2]{Mir}, for example).  We also define
\[ \Lip u(x) := \limsup_{r \to 0} \frac{\sup_{y \in B(x,r)} |u(y)-u(x)|}{r}, \]
the {\em upper Lipschitz constant} of $u$ at $x$.

Since $\Lip_{\loc}(X)$ is dense in $L^1_{\loc} (X)$, we define the
{\em total variation} of $u \in L^1_\loc (X)$ on an open set $W \subset X$ as
\[ 
V(u,W) := \inf \left\{ \liminf_{i\to \infty} \int_W \lip u_i \,d\mu :\, u_i \in \Lip(W),\, u_i \to u \mbox{ in } L^1_\loc(W) \right\}. 
\]
A function $u\in L^1_\loc(X)$  is said to be of {\em locally bounded variation} if $V(u,W)$ is finite for all bounded open $W\subset X$.
A function $u$ is said to be of {\em bounded variation} if $V(u,X)$ is finite.  Let $BV(X)$ denote the set of functions of bounded variation.
For an arbitrary set $A\subset X$, we define
\[
V(u,A):=\inf\{V(u,W):\, A\subset W,\,W\subset X
\text{ is open}\}.
\]
If $V(u,X)<\infty$, then $V(u,\cdot)$ is
a Radon measure on $X$ by \cite[Theorem 3.4]{Mir}, called the variation measure.
In much of current literature on BV functions in metric setting, $V(u,A)$ is also denoted
$\Vert Du\Vert(A)$.
In a significant part of the current literature on BV functions in metric spaces a slightly different
notion of $V(u,W)$ is used, where instead of infimum over $\int_W \lip u_i\, d\mu$ the infimum of the 
integrals $\int_W g_{u_i}\, d\mu$ is considered, where $g_{u_i}$ is the minimal $1$-weak upper gradient of $u_i$, see
for example~\cite{KKLS}. It follows from~\cite{AmbDiM} that these notions all give the same BV class
as well as the \emph{same} BV energy $V(u,W)$ for open sets $W$ (and hence all Borel sets). Thus
we can equivalently define
\[ 
V(u,W) := \inf \left\{ \liminf_{i\to \infty} \int_W g_{u_i} \,d\mu :\, u_i \in \Lip_\loc (W),\, u_i \to u \mbox{ in } L^1(W) \right\}. 
\]

Let $E \subset X$ and let $\chi_E$ denote the characteristic function of $E$.
If $\chi_E$ is of locally bounded variation we say that $E$ is of {\em locally finite perimeter} and if $\chi_E$ is of 
bounded variation, we say that $E$  is of {\em finite perimeter}.  We use $P(E,\cdot) := V(\chi_E, \cdot)$ for 
the {\em perimeter measure}.

The following coarea formula is proven in ~\cite[Proposition 4.2]{Mir}:
if $u\in BV(X)$ and $W\subset X$ is a Borel set, then
\begin{equation}\label{eq:coarea formula}
V(u, W)=\int_{-\infty}^\infty P(\{u>t\},W)\, dt.
\end{equation}

Applying the $1$-Poincar\'e inequality to approximating functions,
we get for any
 $\mu$-measurable set $E\subset X$ and any ball $B=B(x,r)$
the {\em relative isoperimetric inequality}
\begin{equation}\label{eq:relative isoperimetric inequality}
\min\{\mu(B(x,r)\cap E),\mu(B(x,r)\setminus E)\}\le 2 C_P rP(E,B(x, r)),
\end{equation}
see e.g. \cite[Theorem 3.3]{KoLa}.

The $1$-Poincar\'e inequality implies the so-called Sobolev-Poincar\'e inequality, 
see e.g.~\cite[Theorem~4.21]{BB}, from which we get
the following BV version: for every ball $B(x,r)$ and every $u\in L^1_{\loc}(X)$, we have
\begin{equation}\label{eq:BV Sobolev Poincare}
\left(\,\vint_{B(x,r)}|u-u_{B(x,r)}|^{Q/(Q-1)}\,d\mu\right)^{(Q-1)/Q}
\le C r\frac{V(u,B(x,2 r))}{\mu(B(x,2 r))},
\end{equation}
where $Q$ is the exponent from \eqref{eq:homogenous dimension}.

Moreover, we have the following Poincar\'e inequality for
functions vanishing outside a ball.
For any ball $B(x,r)$ with $0<r<\frac{1}{4}\diam X$
and any $u\in L^1(B(x,r))$ with compact support in $B(x,r)$, we have
\begin{equation}\label{eq:poincare inequality with zero bry value}
\int_{B(x,r)} |u|\,d\mu
\le C r V(u,B(x,r));
\end{equation}
this again follows by applying the analogous inequality for Lipschitz
functions (see \cite[Theorem~4.21, Theorem~5.51]{BB})
to an approximating sequence. 

For a set $E\subset X$, the {\em measure-theoretic boundary} is defined
as the set of points of positive upper density for $E$ and $X\setminus E$:
\[ 
\bdry^* E := \left\{ x \in X : \,\limsup_{r \to 0} \frac{\mu(B(x,r) \cap E)}{\mu(B(x,r))}>0 
\quad\mbox{and}\quad \limsup_{r \to 0} \frac{\mu(B(x,r) \setminus E)}{\mu(B(x,r))} >0\right\}.
\]

%



We will also be interested in co-dimension $1$ Hausdorff measures on $X$.
Recall that $\mu$ is \emph{Ahlfors $s$-regular} for
$s>0$ if there is some constant $C_A\ge 1$ such that whenever $x\in X$ and
$0<r<\tfrac 12 \diam X$,
\begin{equation}\label{eq:Ahlfors regular measure}
\frac{r^s}{C_A}\le \mu(B(x,r))\le C_A r^s.
\end{equation}
If $\mu$ is Ahlfors $s$-regular, then the co-dimension $1$ Hausdorff measure
defined below is just
(comparable to) the $(s-1)$-dimensional Hausdorff 
measure.  We do not wish to always assume Ahlfors regularity, however.
We define the {\em co-dimension $1$ Hausdorff measure} 
of a set $E \subset X$ by
\[ \mathcal{H}(E) := \sup_{\delta > 0} \mathcal{H}_\delta (E), \]
where for $\delta>0$,
\[ \mathcal{H}_\delta (E):= 
\inf \left\{ \sum_{i \in I} \frac{\mu(B_i)}{r_i}:\, B_i=B(x_i,r_i),\, r_i \leq \delta,\,
E \subset \bigcup_{i \in I} B_i \right\}. \]

The following density results can be proved similarly as in \cite[Theorem 2.4.3]{AT}.
\begin{lemma}\label{lem:densities and Hausdorff measures}
Let $\nu$ be a Radon measure on $X$, let $A\subset X$, and let $t> 0$.
Then the following hold:
\[
\textrm{if } \ \limsup_{r\to 0}r\frac{\nu(B(x,r))}{\mu(B(x,r))}\ge t\quad\textrm{for all }
x\in A,\textrm{ then}\quad \nu(A)\ge t\mathcal H(A)
\]
and
\[
\textrm{if } \ \limsup_{r\to 0}r\frac{\nu(B(x,r))}{\mu(B(x,r))}\le t\quad\textrm{for all }
x\in A,\textrm{ then}\quad \nu(A)\le C_d t\mathcal H(A).
\]
\end{lemma}

Let $E\subset X$ be a set of finite perimeter. We know that for any Borel set $A\subset X$,
\begin{equation}\label{eq:def of theta}
P(E,A)=\int_{\partial^*E\cap A}\theta_E\,d\mathcal H,
\end{equation}
where
$\theta_E\colon X\to [\alpha,C_d]$ with $\alpha=\alpha(C_d,C_P)>0$,
see \cite[Theorem 5.3]{A} and \cite[Theorem 4.6]{AMP}.
Furthermore, let
\begin{equation}\label{eq:def of Sigma gamma}
\Sigma_\gamma E:= \left\{x\in X: \liminf_{r \to 0} \min \left\{ \frac{\mu(B(x,r) \cap E)}{\mu(B(x,r))}, 
\frac{ \mu (B(x,r) \setminus E) }{\mu(B(x,r))} \right\} \geq \gamma \right\} 
\end{equation}
for a constant $\gamma\in(0,1/2]$ depending only on $C_d,C_P$.
Note that $\Sigma_\gamma E\subset\partial^*E$;
by \cite[Theorem~5.4]{A} we know that conversely,
\begin{equation}\label{eq:mthbdry minus Sigmagamma}
\mathcal H(\partial^*E\setminus \Sigma_\gamma E)=0.
\end{equation}

\begin{lemma}\label{lem:perim-growth-bounds}
Let $E\subset X$ be a set of finite perimeter. Then
for $\mathcal H$-a.e. $x\in\partial^*E$
(and thus $P(E,\cdot)$-a.e. $x\in \partial^*E$),
\begin{equation}\label{eq:density of perimeter}
\frac{\gamma}{C_P} \le\liminf_{r\to 0}\frac{P(E,B(x,r))}{\mu (B(x, r))/r} \le
\limsup_{r\to 0}\frac{P(E,B(x,r))}{\mu (B(x, r))/r} \le C_d.
\end{equation}
\end{lemma}

\begin{proof}
The first inequality holds for every $x\in\Sigma_{\gamma} E$ by the relative
isoperimetric inequality~\eqref{eq:relative isoperimetric inequality}.
To show the second inequality,
note that if $A\subset \partial^*E$ and $\eps>0$ are such that
\[
\limsup_{r\to 0}r\frac{P(E,B(x,r))}{\mu (B(x, r))} \ge C_d+\eps
\]
for all $x\in A$, then by the first part of
Lemma~\ref{lem:densities and Hausdorff measures} and by~\eqref{eq:def of theta}, we have
$P(E,A)\ge (C_d+\eps)\mathcal H(A)$.
However, according to~\eqref{eq:def of theta}, we have
$P(E,A)\le C_d\mathcal H(A)$. Thus we must have $\mathcal H(A)=0$.
\end{proof}

The lower and upper approximate limits of a function $u$ on $X$
are defined respectively by
\[
u^{\wedge}(x):
=\sup\left\{t\in\R:\,\lim_{r\to 0}\frac{\mu(B(x,r)\cap\{u<t\})}{\mu(B(x,r))}=0\right\}
\]
and
\[
u^{\vee}(x):
=\inf\left\{t\in\R:\,\lim_{r\to 0}\frac{\mu(B(x,r)\cap\{u>t\})}{\mu(B(x,r))}=0\right\}.
\]
The {\em jump set} $S_u$ is defined to be the set where $u^{\wedge}<u^{\vee}$.

By \cite[Theorem 5.3]{AMP}, the variation measure of a $BV$ function
can be decomposed into the absolutely continuous and singular part, and the latter
into the Cantor part and jump part, as follows. Given
$u\in BV(X)$, we have for any Borel set $A\subset X$
\[
\begin{split}
V(u,A) &=V_a(u,A)+V_s(u,A)\\
&=V_a(u,A)+V_c(u,A)+V_j(u,A)\\
&=\int_{A}g\,d\mu+V_c(u,A)+\int_{A\cap S_u}\int_{u^{\wedge}(x)}^{u^{\vee}(x)}\theta_{\{u>t\}}(x)\,dt\,d\mathcal H(x),
\end{split}
\]
where $g\in L^1(X)$ is the density of the absolutely continuous part and the functions $\theta_{\{u>t\}}$ 
are as in~\eqref{eq:def of theta}.

We denote by $BV_c(X)$ the class of BV functions with compact support in $X$.

\begin{definition}
	We say that $u\in BV(X)$ is a {\em function of least gradient} if
	for all $\varphi\in  BV_c(X)$,
	\begin{equation}\label{eq:definition of 1minimizer}
	V(u,\supp\varphi)\le V(u + \varphi,\supp\varphi).
	\end{equation}
\end{definition}

\section{Pointed measured Gromov-Hausdorff limits}\label{sec:PMGHC}

In this section we consider tangent spaces of a metric space at a given point.
For this, we first need to specify what is meant by the convergence of metric spaces. Existing literature has some slightly different definitions and diverging terminology; here we describe them and provide brief explanation on how these are
equivalent.

\begin{definition}\label{pGH} 
We say that the sequence of pointed metric spaces $(Y_n, d_n, y_n)$ {\it converges in the 
pointed Gromov-Hausdorff distance} to the space $(Y_\8,d_\8, y_\8)$ if
for each positive integer $n$ there is a map $\phi_n:Y_\8\to Y_n$ so that 
$\phi_n(y_\8)=y_n$, and for each $R>0$ and $\epsilon>0$ there is a positive integer $N_{\epsilon, R}$ such that whenever $k\ge N_{\epsilon,R}$,
 we have 
 \begin{enumerate}
 \item $\sup_{x,y\in B_{Y_\8}(y_\8,R)}|d_{Y_k}(\phi_k(x),\phi_k(y))-d_{Y_\8}(x,y)|<\epsilon$,
 \item $B_{Y_k}(y_k,R-\epsilon)\subset \bigcup_{y\in \phi_k(B_{Y_\8}(y_\8,R))}B_{Y_k}(y,\epsilon)$.
\end{enumerate}
\end{definition}

Note that these maps are not required to be continuous, or even measurable. It is possible to modify $\phi_n$ to be measurable, but this is technical, and not necessary for our presentation below.

\begin{remark}\label{rmk:control-eps} 
The above definition is compatible with those of~\cite{BBI, Kei03}. In~\cite[Definition~8.1.1]{BBI}
and~\cite[Chapter~11]{HKST}
the following definition of pointed Gromov-Hausdorff convergence was considered:
For all $r>0$ and all $0 < \epsilon < r$ there exists an 
$n_0 = n_0(r, \epsilon)$ such that for all $n \geq n_0$ there exist functions
$\phi_n^\epsilon\colon B_{Y_\8}(y_\8, r) \to Y_n$ with
\begin{enumerate}
\item $\phi_n^\epsilon (y_\8) = y_n$, \\
\item $| d_n (\phi_n^\epsilon (x), \phi_n^\epsilon (y)) - d_\8 (x,y) | < \epsilon$ for all $x,y \in B_{Y_\8}(y_\8, r)$, \\
\item $B_{Y_n} (y_n, r - \epsilon) \subset \bigcup_{y \in \phi_n^\epsilon (B_{Y_\8}(y_\8, r))} B_{Y_n}(y, \epsilon)$. \\
\end{enumerate}
See~\cite{Her} for more on pointed Gromov-Hausdorff convergence. 
To see the compatibility between these two definitions we note that 
the scales $R$ and $\epsilon$ play the role of localizing the convergence of the tangent spaces. Thus, 
the second notion is implied by the first, as seen by the choice $\phi_n^\epsilon:=\phi_n$.
Conversely, given $\phi_n^\epsilon$, choosing a sequence of $R_n$ monotonically increasing to $\infty$ and $\epsilon_n$ monotonically decreasing to $0$,
we can even choose $\phi_n^\epsilon$ to be independent of $\epsilon$ and $r$; 
hence the equivalence of the notion of~\cite{BBI} with ours. However, in proofs it is often easier to work 
with the localized versions $\phi_n^\epsilon$, since it avoids this additional diagonal argument. Where we wish to 
use globally defined functions, we use $\phi_n$. These are interchangeable.

The notion considered in~\cite{Kei03} is also equivalent to the above. Since this notion of~\cite[Definition~2 and Definition~7]{Kei03} is
also useful in this paper, especially in defining notions of weak convergence of measures to tangent spaces, we now provide
that definition as well. According to~\cite{Kei03}, the sequence $(Y_n,d_n,y_n)$ converges to  a proper space $(Y_\8,d_\8,y_\8)$ if there is a
proper metric space $(Z,d_Z)$ and a point $z_0\in Z$, an isometric embedding $\iota:Y_\8\to Z$, and 
for each $n\in\N$ there is an isometric embedding $\iota_n:Y_n\to Z$, such that $\iota(y_\8)=z_0=\iota_n(y_n)$ and for each $R>0$,
 \begin{enumerate}
 \item $\lim_{n\to\8} \sup_{y\in B_{Y_n}(y_n,R)} \text{dist}_Z(\iota_n(y),\iota(Y_\8))=0$,
 \item $\lim_{n\to\8} \sup_{z\in B_{Y_\8}(y_\8,R)}\text{dist}_Z(\iota(z),\iota(Y_n))=0$.
\end{enumerate}
From this definition we see that whenever $R,\eps>0$ there is some positive integer $N_{\eps,R}$ such that whenever $n>N_{\eps,R}$,
for each $x,y\in B_{Y_n}(y_n,R)$ we can find $\widehat{x},\widehat{y}\in B_{Y_\8}(y_\8,R+\eps)$ such that
\[
\max\{d_Z(\iota_n(x),\iota(\widehat{x})), d_Z(\iota_n(y),\iota(\widehat{y}))\}<\eps, \ \ \ 
|d_{Y_n}(x,y)-d_{Y_\8}(\widehat{x},\widehat{y})|<3\eps.
\]
We also have that for $R>0$ and $\eps>0$ there is some positive integer $N_{\eps,R}$ such that for $n>N_{\eps,R}$,
whenever $x,y\in B_{Y_\8}(y_\8,R)$ there exist $x_n,y_n\in B_{Y_n}(y_n,R+\eps)$ such that
\[
\max\{d_Z(\iota_n(x_n),\iota(x)), d_Z(\iota_n(y_n),\iota(y))\}<\eps, \ \ \ 
|d_{Y_n}(x_n,y_n)-d_{Y_\8}(x,y)|<3\eps.
\]
This shows that the definition of~\cite{Kei03} implies our definition above. The fact that our definition implies the one of~\cite{Kei03}
comes from the construction of the ambient space $Z$ found in~\cite{Her}, where the space $Z$ should be considered to be the
completion of the ``disjoint union" space $Y$ found in~\cite[Section~4.1.1]{Her}.

Indeed, we can construct the maps $\iota_n$
and $\iota$ from the maps $\phi_n$ and vice versa so that the following compatibility condition between these two
classes of maps is satisfied: For all $r>0$,
\begin{equation}\label{eq:compatibile-phi-iota}
\lim_{n\to\infty}\sup_{y\in B_{Y_\8}(y_\8,r)}d_Z(\iota_n\circ\phi_n(y),\iota(y))=0.
\end{equation}

For simplicity, and avoiding modifying the space $Z$, as well as the approximating maps $\phi_n$, we will generally fix them throughout the exposition below. In order to define other notions, such as convergence of points, curves and functions, passing to a subsequence in $n$ may be necessary. However, this subsequence will be of $n$ and will not require coming up with new $\phi_n$ or embedding space  $Z$. 

In the light of the above discussion, we can say that a sequence, $z_n\in Y_n$, converges to $z\in Y_\8$ if
$\lim_{n\to\infty}d_Z(\iota_n(z_n),\iota(z))=0$, and see that every $z\in Y_\8$ is a limit of a sequence $z_n \in Y_n$
as here. By a not-terrible abuse of notation we denote this by
\[
\lim_{n\to\infty}z_n=z.
\]
\end{remark}

Next, we define pointed {\em measured} Gromov-Hausdorff convergence.
For this, we use the embeddings described in the above remark.
First consider a sequence of Borel measures $\nu_n$ on a metric space $Z$.
The measures $\nu_n$ converge weakly* to a Borel measure $\nu$ on $Z$ if
\[ 
\int_Z \phi \,d\nu_n \to \int_Z \phi \,d\nu 
\]
as $n \to \infty$ for all boundedly supported continuous functions $\phi$ on $Z$.
We denote this convergence by $\nu_n \overset{*}{\rightharpoonup} \nu$.

To define measured Gromov-Hausdorff convergence,
we consider the push-forward measures 
\[ 
\iota_{n,*} \nu_n(A) := \nu_n (\iota_n^{-1}(A)). 
\]
We say that the sequence of Radon measures $\nu_n$ on $Y_n$ converges to a Radon measure 
$\nu_\8$ on $Y_\8$, denoted 
$\nu_n\overset{*}\rightharpoonup\nu_\8$, if $\iota_{n,*}\nu_n\overset{*}\rightharpoonup\iota_*\nu_\8$ on $Z$. 

\begin{definition}\label{def:pmGH} 
We say that a sequence of pointed metric measure spaces $(Y_n,d_n,y_n,\nu_n)$ converges pointed measured Gromov-Hausdorff to a space 
$(Y_\8,d_\8,y_\8,\nu_\8)$, if the sequence converges in the pointed Gromov-Hausdorff sense, and 
\[
\nu_n\overset{*}{\rightharpoonup} \nu_\8.
\]
\end{definition}

Since $Z$ is a proper metric space, it follows that whenever $\sup_n\nu_n(\iota_n^{-1}(Z))<\infty$, there is a
subsequence $\nu_{n_k}$ and a Radon measure $\widehat{\nu_\8}$ on $Z$
such that $\iota_{n_k,*}\nu_{n_k}\overset{*}\rightharpoonup\widehat{\nu_\8}$ in $Z$. This limit measure must have support 
in $\iota(Y_\8)$, since the support of $\nu_\8$ is contained in the limit of the supports of $\nu_{n_k}$. 
Indeed, given $\eps>0$ and a radius $R>0$, we know that for large $n$ the set $\iota_n(B_{Y_n}(y_n,R))$ is in an 
$3\eps$-neighborhood of $\iota(Y_\8)$. Recall that $Y_\8$ is a proper metric space.
We call such measures $\nu_\8$ limit measures of the sequence $\nu_{n_k}$, and they may depend on the choice 
of the subsequence; the full sequence $\nu_n$ may not converge to $\nu_\8$. In the proofs below, we will 
always pass to the subsequence where this limit holds.

\begin{lemma} 
In the above situation, $\widehat{\nu_\8}(Z\setminus\iota(Y_\8))=0$, and hence
there is a Radon measure $\nu_\8$ on $Y_\8$ such that $\widehat{\nu_\8}=\iota_* \nu_\8$.
\end{lemma}




\begin{definition}\label{defn:tangent-space}
Let $x\in X$ and let $r_n>0$ with $r_n\to 0$.
Define the sequence of
scaled metrics $d_n$ on $X$ by
\[
d_n(y,z):=\frac{d(y,z)}{r_n},
\]
and the scaled measures
\[
\mu_n:=\frac{1}{\mu(B(x,r_n))}\, \mu.
\]
If the sequence
{$(X_n, d_n, x, \mu_n):=(X,d_n,x,\mu_n)$}
converges to 
$(X_\infty, d_\infty, x_\infty,\mu_\infty)$ in the pointed measured
Gromov-Hausdorff sense, then we say that $X_\infty$
is a \emph{tangent space to $X$ 
at $x$, with tangent measure $\mu_\infty$.}
\end{definition}

We know that if $\mu$ is doubling and $X$ is a 
complete geodesic space, then by passing to a
subsequence of $(X,d_n,x,\mu_n)$ if necessary, we will always have a tangent metric measure 
space as above, which is also geodesic,
see \cite[Section~6]{Gro} or the discussion
in~\cite[Section 11]{HKST}. 
Note that
points of distance less 
than $r_n$ from $x$ in $(X,d)$ are, in the space $X_n$,
at distance less than 1 from $x$,
and that the ball $B(x, r_n)$ has $\mu_n$-measure 1. The tangent space may be non-unique, 
and depends on the subsequence chosen.

From the work of~\cite{Kei03} we know that if $\mu$ is doubling and supports a $1$-Poincar\'e inequality,
then for every $x\in X$, all the corresponding tangent
spaces have the tangent measure
be doubling and support a $1$-Poincar\'e inequality, with the doubling and Poincar\'e constants depending
quantitatively only on the corresponding constants for $X$, see also~\cite{HKST}.  A 
proof of this first appeared in the work~\cite{Kei03} of Keith, but he reports in~\cite{Kei03} that it was 
independently found by himself, Koskela, and Cheeger.


We will fix the following notion of a limit of functions. 

\begin{definition}\label{def:conv-functions}
We say that a function $u_\8$ on $X_\8$ {\it is a limit of} $u_n$ (with $u_n$ a function on $X_n$)
if there exists some subsequence $n_k$ and $\epsilon_k \searrow 0$ such that for all $r >0$
\begin{equation}\label{zoomlimit} 
\lim_{k \to \infty}\| u_\8 - u_{n_k}\circ \phi^{\epsilon_k}_{n_k} \|_{L^\8 (B_{X_\8}(x_\8, r))}=0.
\end{equation}
\end{definition}

This is equivalent to the following definition of limits using globally defined maps $\phi_k$:
\begin{equation} \label{eq:zoomlimit-2}
\lim_{k \to \infty}\| u_\8 - u_{n_k}\circ \phi_{k} \|_{L^\8 (B_{X_\8}(x_\8, r))}=0. 
\end{equation}
Given a sequence of functions $u_n$ that are uniformly Lipschitz 
and locally uniformly bounded, one can extract some subsequence and find a limit function $u_ \8$. For example, by considering first 
the convergence of the values of $u_n$ along some sequence of ever-denser nets of $X_n$ converging to nets in 
$X_\8$ and diagonalizing. This is easier to see in terms of the globally defined $\phi_n$, but can also 
be done via a detailed diagonal argument with $\phi^\epsilon_n$. This idea can also be seen from the point of view of the definition of~\cite{Kei03}.
For each $n$ we can find a Lipschitz extension $\widehat{u_n}$ of $u_n\circ\iota_n^{-1}$ from $\iota_n(X_n)$ to $Z$.
This sequence forms an equibounded and equicontinuous sequence of functions in $Z$ which, being a proper space, lends itself
to an application of the Arzel\`a-Ascoli theorem. Thus we may find a subsequence of $\widehat{u_n}$ that converges locally
uniformly to a Lipschitz function $\widehat{u_\8}$ on $Z$. We can now choose $u_\8=\widehat{u_\8}\circ \iota$.  
That this choice of $u_\8$ is a limit of $u_n$ follows from the compatibility condition~\ref{eq:compatibile-phi-iota}.

The notion of limit of functions as given above is concordant with the notion of limit of measures.
If $u_n$ and $u_\8$ are uniformly Lipschitz and $u_\8$ is a limit of $u_{n_k}$, then along the same subsequence 
$\nu_{n_k} := u_{n_k}\,d\mu_{n_k} \overset{*}\rightharpoonup \nu_\8 := u_\8\, d\mu_\8$.

As before, these functions depend on the subsequence chosen. In fact, there are several dependencies. The very tangent 
spaces depend on the subsequence of blow ups $r_n \searrow 0$, as well as on the mappings $\iota_n, \phi^{\epsilon_k}_n$ 
which are not canonical. We will always assume that all these choices are given to us. 



\section{Asymptotics at approximate continuity points and generalized linear functions}

The goal of this section is to study the asymptotic behavior of BV functions at points outside
the jump and Cantor parts of their variation measures. We start with the following handy lemma.

\begin{lemma}\label{lem:equal-energy}
Let $u,v\in BV(X)$. Suppose $E\subset X$ is a Borel set such that for each $x\in E$ we have
\[
\lim_{r\to 0^+}\frac{\mu(B(x,r)\cap E)}{\mu(B(x,r))}=1,
\]
and $u(x)=v(x)$. Then $V(u-v,E)=0$ and so for each $A\subset E$ we have
$V(u,A)=V(v,A)$. 
\end{lemma}

\begin{proof}
From the above, we know that for each $t\in\R$ the set
$E\cap\partial^*E_t$ is empty, where $E_t=\{x\in X\, :\, u(x)-v(x)>t\}$.
Therefore by the coarea formula~\eqref{eq:coarea formula}
and by~\eqref{eq:def of theta}, the claim follows. 
\end{proof}

In the Euclidean setting we know that for $\mathcal L^n$-a.e. $x\in \R^n$
(where $\mathcal L^n$ denotes the $n$-dimensional Lebesgue measure)
a BV function converges under blow-up to a linear 
function (see e.g. \cite[Theorem 3.83]{AFP}).
In the metric setting, for $p>1$
the notion of linear function is interpreted as a function that is
constant or else satisfies the following
two properties: (a) the image of $X$ under the function is $\R$, and (b) the minimal
$p$-weak upper gradient of the function 
is constant (and given that we have a Poincar\'e inequality, this constant should be
non-zero if the function is not the constant 
function); see for example~\cite{Che}.  It was shown in~\cite{Che} that given a Lipschitz function, any 
asymptotic limit of that function
at almost every point yields such a linear function on the corresponding tangent 
space, which we defined in Section~\ref{sec:PMGHC}.
In the case $p=1$, which is the natural setting 
for BV functions, we will prove that the asymptotic limits are the so-called
\emph{generalized linear functions} on the tangent spaces.

For $g \in L^1_{\loc}(X)$ nonnegative and $R>0$, we define the
restricted maximal function of $g$ at $x\in X$ by
\[ 
\mathcal{M}_R g(x):= \sup_{0<r\le R}\, \vint_{B(x,r)} g \,d\mu.
\]
The maximal function of a Radon measure $\nu$ is defined similarly by
\[
\mathcal{M}_R \nu(x):= \sup_{0<r\le R} \frac{\nu(B(x,r))}{\mu(B(x,r))}.
\]
Recall that if $u\in BV(X)$, then
\[
dV(u,\cdot) =g\, d\mu+dV_s(u,\cdot),
\]
where $g\in L^1(X)$ is the Radon-Nikodym derivative of $V(u,\cdot)$
with respect to $\mu$ and $V_s(u,\cdot)$ is the singular part.

\begin{proposition}\label{prop:Lipschitz continuity around a point}
Let $u\in BV(X)$. Then for $\mu$-a.e. $x\in X$ for which
$g(x)>0$ there exists $R>0$ and a set $A_x\not\ni x$ with density $0$ 
at the point $x$ such that $u|_{B(x,R)\setminus A_x}$ is Lipschitz with constant $Cg(x)$.
For $\mu$-a.e. $x\in X$ for which
$g(x)=0$,
for each $\delta>0$ there is a set $A_{x,\delta}\not\ni x$ with density $0$ at $x$ such that 
$u|_{B(x,R)\setminus A_{x,\delta}}$ is Lipschitz with constant $\delta$.
\end{proposition}

\begin{proof}
We follow the proof of \cite[Proposition 13.5.2]{HKST}.
For $\mu$-a.e. $x\in X$, we have
\begin{equation} \label{eq:Lebesgue point}
\lim_{r\to 0}\,\vint_{B(x,r)}|g-g(x)|\,d\mu= 0\quad\textrm{and}\quad \lim_{r\to 0}\frac{V_s(u,B(x,r))}{\mu(B(x,r))}=0
\end{equation}
by the Lebesgue-Radon-Nikodym theorem,
see e.g. \cite[Section 3.4]{HKST}.
Fix such $x\in X$ for which also $g(x)>0$.

Let $R>0$ and let $y,z\in B(x,R)$ be Lebesgue points of $u$. 
For nonnegative integers $i$ we set $B_i:=B(y,2^{-i}d(y,z))$,
and for negative integers $i$ we set
$B_i:=B(z,2^id(y,z))$. Then, by the doubling property of $\mu$ and the  Poincar\'e inequality,
\begin{align*}
|u(y)-u(z)|
	&\le \sum_{i\in\Z}|u_{B_i}-u_{B_{i+1}}|\\
	&\le C \sum_{i\in\Z}\,\vint_{2B_{i}}|u-u_{2B_{i}}|\,d\mu\\
  	&\le C d(y,z)\sum_{i\in\Z}2^{-|i|}\frac{V(u,2 B_i)}{\mu(2 B_i)}\\
  	&= C d(y,z)\sum_{i\in\Z}2^{-|i|}\left(\,\vint_{2 B_i}g\, d\mu+\frac{V_s(u,2 B_i)}{\mu(2 B_i)}\right)\\
  	&\le C d(y,z)\sum_{i\in\Z}2^{-|i|}\left(\,\vint_{2 B_i}|g-g(x)|\, d\mu+\frac{V_s(u,2 B_i)}{\mu(2 B_i)}\right)
    +C d(y,z)g(x).
\end{align*}
For $s>0$, set
\[
\tau_s:=\sup_{0<r\le s}\left(\,\vint_{B(x, r)}|g-g(x)|\,d\mu
+\frac{V_s(u,B(x, r))}{\mu(B(x, r))}\right).
\]
Note that since $\mu$ is doubling, for any $s>0$
and any Radon measure $\nu$ we have
\[\mathcal M_{s} \nu (y) \leq C \mathcal M_{2d(x,y)} \nu (y) 
+ C \mathcal M_{2s} \nu(x).\]
Applying this with $s=2d(y,z)<4R$ in the second inequality below,
we get
\begin{align}\label{eq:telescoping argument in Lipschitz proof}
|u(y)-u(z)| 
&\le C d(y,z)[g(x)+\mathcal M_{2d(y,z)}|g-g(x)|(y)+\mathcal M_{2d(y,z)}|g-g(x)|(z)\notag\\
&\quad+\mathcal M_{2d(y,z)}V_s(u,\cdot)(y)+\mathcal M_{2d(y,z)}V_s(u,\cdot) (z)] \notag\\
&\le C d(y,z)[g(x)+\mathcal M_{2d(x,y)}|g-g(x)|(y)+\mathcal M_{2d(x,z)}|g-g(x)|(z)\notag\\
&\quad+\mathcal M_{2d(x,y)}V_s(u,\cdot)(y)+\mathcal M_{2d(x,z)}V_s(u,\cdot) (z) + \mathcal{M}_{8R}[g-g(x)](x)+\mathcal{M}_{8R} V_s(u,\cdot)(x) ]\notag\\
&\leq C d(y,z)[g(x)+\mathcal M_{2d(x,y)}|g-g(x)|(y)+\mathcal M_{2d(x,z)}|g-g(x)|(z)\notag\\
&\quad+\mathcal M_{2d(x,y)}V_s(u,\cdot)(y)+\mathcal M_{2d(x,z)}V_s(u,\cdot) (z) + \tau_{8R}].
\end{align}

We only consider $R>0$ to be small enough so that $\tau_{8R}<g(x)$
{(here we need the fact that $g(x)>0$)}.
We choose a sequence of radii $R_M\searrow 0$ as $M\to\infty$ such
that $2^M\tau_{8 R_M}<g(x)$ for each $M\in\N$.
Next let $A_M$ be the set of all points $y\in B(x,R_M)$ such
that for some $0<r\le 2d(x,y)$,
\[
\vint_{B(y,r)}|g-g(x)|\,d\mu+\frac{V_s(u,B(y,r))}{\mu(B(y,r))}
>2^M\tau_{8 R_M}.
\]
For each $y\in A_M$ there is a ball $B(y,r_y)$ with
$0<r_y\le 2d(x,y)<2R_M$ such that
the above inequality holds, and so the family $\{B(y,r_y)\}_{y\in A_M}$ is a cover of $A_M$. By the $5$-covering theorem
we can extract a countable, pairwise disjoint subfamily $\mathcal{G}$ of the above family such that 
$A_M\subset\bigcup_{B\in\mathcal{G}}5B$.
If $\tau_{8 R_M}=0$, then $\mu(A_M)=0$; else
we see by the doubling property of $\mu$ that
\begin{align*}
\mu(A_M)&\le C\sum_{B\in\mathcal{G}}\mu(B)\\
    &\le \frac{C}{2^M\tau_{8 R_M}}\sum_{B\in\mathcal{G}}\left(\int_B|g-g(x)|\, d\mu+V_s(u,B)\right)\\
    &\le \frac{C}{2^M\tau_{8 R_M}}\left(\int_{B(x,4R_M)}|g-g(x)|\, d\mu+V_s(u,B(x,4R_M))\right)\\
    &\le \frac{C\mu(B(x,4R_M))}{2^M\tau_{8 R_M}}
        \left(\,\vint_{B(x,4R_M)}|g-g(x)|\, d\mu+\frac{V_s(u,B(x,4R_M))}{\mu(B(x,4R_M))}\right)\\
    &\le \frac{C\mu(B(x,R_M))}{2^M\tau_{8 R_M}}\, \tau_{4 R_M}
    \le \frac{C \mu(B(x,R_M))}{2^M}.
\end{align*}
We can add to each $A_M$ all the non-Lebesgue points of $u$
different from $x$,
without adding measure.
Now note that $x\not\in A_M$ and that
by~\eqref{eq:telescoping argument in Lipschitz proof}, $u$ is $Cg(x)$-Lipschitz in $B(x,R_M)\setminus A_M$. By choosing
\[
A_x:=\bigcup_{M=1}^{\infty} A_M\setminus B(x,R_{M+1}), 
\]
we see that $u$ is $Cg(x)$-Lipschitz in $B(x,R_1)\setminus A_x$.
Indeed, if $y,z\in B(x,R_1)\setminus A_x$ such that $y\ne x\ne z$, then there are positive integers $M_1$ and
$M_2$ such that $y\in B(x,R_{M_1})\setminus B(x,R_{M_1+1})$ with $y\not\in A_{M_1}$ and
$z\in B(x,R_{M_2})\setminus B(x,R_{M_2+1})$ with $z\not\in A_{M_2}$.
We can assume that $M_2>M_1$.
It then follows 
from~\eqref{eq:telescoping argument in Lipschitz proof} that
\begin{equation}\label{eq:Lip-control-with-Cgx}
|u(y)-u(z)|\le C\, d(y,z)[ g(x)+2^{M_1}\tau_{8R_{M_1}}+2^{M_2}\tau_{8R_{M_2}}+{\tau_{8R_{M_1}}}]\le 4C\, g(x)\, d(y,z).
\end{equation}
Therefore $u$ is $Cg(x)$-Lipschitz continuous in $B(x,R_1)\setminus (A_x\cup\{x\})$. The fact that $u$ is approximately
continuous at $x$, together with the fact that $A_x$ has lower density zero at $x$ (see the argument below), tells us that 
$u$ is $Cg(x)$-Lipschitz continuous in $B(x,R_1)\setminus A_x$.
%
%

Moreover,
\[
\frac{\mu(A_x\cap B(x,R_{M_0}))}{\mu(B(x,R_{M_0}))}
  \le C\sum_{M=M_0}^\infty\frac{\mu(A_M)}{\mu(B(x,R_{M}))}\le C \sum_{M=M_0}^\infty 2^{-M}\to 0\quad\text{as }M_0\to\infty.
\] 
This guarantees that $A_x$ has {\em lower} density $0$ at $x$.
On the other hand, by the choice of 
the covering of $A_M$ by balls $B(y,r_y)$ with radius $r_y\le 2d(x,y)$, 
in the estimate for $\mu(A_M)$ obtained above we can in fact
obtain for any $0<r\le R_{M}$ that
\[
\mu(A_M\cap B(x,r))\le \frac{C}{2^M}\mu(B(x,4r))\le 
\frac{C C_d^2}{2^M}\mu(B(x,r)).
\]
This guarantees that $A_x$ has density $0$ at $x$.
Choosing $R=R_1$, we have proved the first claim.

Finally, if $\{x\in X :\, g(x)=0\}$ has positive measure, then the above argument gives that for $\mu$-a.e.~$x$ in this set, 
for every $\delta>0$ there exists $A_{x,\delta}$ with $x\not\in A_{x,\delta}$ and $A_{x,\delta}$ is of density $0$
at $x$ such that $u\vert_{B(x,R)\setminus A_{x,\delta}}$ is $\delta$-Lipschitz.\end{proof}


\begin{lemma}\label{lem:Lebesgue point lemma}
	Let $u\in BV(X)$. For $\mu$-a.e. $x\in X$ the following holds:
	if $A\subset X$ has density $0$ at $x$, then
	\[
	\lim_{r \to 0}\frac{1}{r}\frac{1}{\mu(B(x,r))}\int_{A \cap B(x,r)} |u-u(x)|\,d\mu = 0.
	\]
\end{lemma}

\begin{proof}
	Excluding a $\mu$-negligible set, we can
	take a Lebesgue point $x$ of $u$ such that
	(just as in \eqref{eq:Lebesgue point})
	\begin{equation}\label{eq:assumption on density of V}
	\lim_{r\to 0}\frac{V(u,B(x,r))}{\mu(B(x,r))}=g(x).
	\end{equation}
By H\"older's inequality,
\[ 
\frac{1}{r}\frac{1}{\mu(B(x,r))}
\int_{A \cap B(x,r)} |u-u(x)|\,d\mu 
\le \frac{1}{r}
\left(\,\vint_{B(x,r)}|u-u(x)|^{Q/(Q-1)}\,d\mu\right)^{(Q-1)/Q}
\left(\frac{\mu(A\cap B(x,r))}{\mu(B(x,r))}\right)^{1/Q}.
\] 
Since $x$ is a Lebesgue point of $u$, by the Sobolev-Poincar\'e
inequality \eqref{eq:BV Sobolev Poincare},
\begin{align*}
\left(\,\vint_{B(x,r)}|u-u(x)|^{Q/(Q-1)}\,d\mu\right)^{\frac{Q-1}{Q}}
&\le\left(\,\vint_{B(x,r)}|u-u_{B(x,r)}|^{Q/(Q-1)}\,d\mu\right)^{\frac{Q-1}{Q}}\\
&\qquad+\sum_{j=1}^{\infty}
\left(\,\vint_{B(x,r)}|u_{B(x,2^{-j+1}r)}-u_{B(x,2^{-j}r)}|^{Q/(Q-1)}
\,d\mu\right)^{\frac{Q-1}{Q}}\\
&\le C r\frac{V(u,B(x,r))}{\mu(B(x,r))}
+\sum_{j=1}^{\infty}
|u_{B(x,2^{-j+1}r)}-u_{B(x,2^{-j}r)}|\\
&\le C r\frac{V(u,B(x,r))}{\mu(B(x,r))}
+C
\sum_{j=1}^{\infty}2^{-j+1}r
\frac{V(u,B(x,2^{-j+1}r))}{\mu(B(x,2^{-j+1}r))}\\
&\le C r\mathcal M_{r}V(u,\cdot)(x).
\end{align*}
Thus we get
\[
\frac{1}{r}\frac{1}{\mu(B(x,r))}\int_{A \cap B(x,r)} |u-u(x)|\,d\mu
\le C \mathcal M_{r}V(u,\cdot)(x)\left(\frac{\mu(A
	\cap B(x,r))}{\mu(B(x,r))}\right)^{1/Q}.
\]
Note that by \eqref{eq:assumption on density of V}, $\lim_{r\to 0}\mathcal M_{r}V(u,\cdot)(x)=g(x)<\infty$,
and so by the fact that $A$ has density $0$ at $x$, we get the conclusion.
\end{proof}

\begin{definition}
Let $v$ be a real-valued function on a metric space $(Z,d)$. The \emph{oscillation} of $v$ in a ball $B(x,r)$ is
\[
\underset{(x,r)}{\var} v:=\sup_{y\in B(x,r)}\frac{|v(y)-v(x)|}{r}.
\]
We also set
\[
\LIP v:=\sup_{y,z\in Z\, :\, y\ne z}\frac{|v(y)-v(z)|}{d(y,z)}.
\]
\end{definition}

Observe that $\underset{(x,r)}{\var} v\le \LIP v$.


We now return to the sequence $X_n$ of zoomed-in versions of $X$,
as defined in Section \ref{sec:PMGHC}.
For $u\in BV(X)$, we wish to study the limit of the functions
\[ 
u_n(y) := \frac{u(y) - u(x)}{r_n}
\]
that are defined on $X_n$.
%
Suppose the point $x$ satisfies the conclusion of 
Proposition~\ref{prop:Lipschitz continuity around a point}.  Zooming in 
and defining $u_n$ as above, we note that $u$ is only known to be Lipschitz continuous on $B(x, R) \setminus A_{x}$.  This 
poses a problem for studying the supposed limit function $u_\8$.  Though the set $A_{x}$ from 
Proposition~\ref{prop:Lipschitz continuity around a point} has zero density at point $x$, it could still be very much 
in the image of the functions $\phi_n^\epsilon$ used for comparing $u_n$ with $u_\infty$ (note that the $\phi_n^\epsilon$ are not 
necessarily continuous).  This could have the effect of a limit function $u_\8$ having little to do with the values of $u$ outside 
$A_{x}$, a set which is quantitatively marginal as we have shown.  It seems prudent to search for a limit 
function $u_\8$ that reflects the values of $u$ on $B(x, R) \setminus A_{x}$ if we want to explore any properties 
of this limit function.  With that in mind, we make the following definition:

\begin{definition}
For $x\in X$ for which the conclusion of 
Proposition~\ref{prop:Lipschitz continuity around a point} holds,
we say that the functions $\phi_n$ in the definition of pointed measured Gromov-Hausdorff convergence 
are {\it adapted to} $u$ if
\[ 
\phi_n (B_{X_\8} (x_\8, R)) \cap A_{x} = \emptyset 
\]
for all {$n \ge N_{\epsilon,R}$.}
\end{definition}

Thanks to the following lemma we know that whenever
$(X_\8,  d_\8,x_\8,\mu_\8)$ is a tangent space to $X$ at $x$
as in Definition~\ref{defn:tangent-space}, we can always find a subsequence of the sequence $(X_n,d_n,x,\mu_n)$
such that the corresponding maps $\phi_n$ are adapted to $u$.

\begin{lemma}
Suppose $u \in BV(X)$ and $x\in X$ is a point for which the conclusion of 
Proposition~\ref{prop:Lipschitz continuity around a point} holds.
If $(X_\8, d_\8,x_\8, \mu_\8)$ is a pointed measured Gromov-Hausdorff 
limit of $(X, d_n, x, \mu_n)$ 
for some positive sequence $r_n \to 0$, then 
there exist functions $\phi_n$
that are adapted to $u$ at $x$.
\end{lemma}

\begin{proof}
Assume for simplicity that
{$R=1/2$}
and fix $0<\epsilon \le 1$. By the doubling condition, we have that
\begin{equation}\label{eq:Ax0 having density zero}
\frac{\mu (B(x, 2s) \cap A_{x})}{\mu (B(x, s))} \to 0 \quad\textrm{as }s\to 0.
\end{equation}
By \eqref{eq:homogenous dimension}, there exists $Q>1$ such that
whenever $y\in B(x,s)$ and $0<t\le s$,
\[
 \frac{1}{C} \left(\frac{t}{s}\right)^{Q}\le \frac{\mu(B(y,t))}{\mu(B(x,s))}.
\]
It follows that if $y\in B(x,r_n)$ (i.e. $y\in B_n(x,1)$), then
\[
\frac{1}{C} \epsilon^{Q}\le \frac{\mu(B(y,\epsilon r_n))}{\mu(B(x,r_n))}=\mu_n(B_n(y,\epsilon)),
\]
where $B_n$ is the ball in the metric $d_n=r_n^{-1} d$.
On the other hand, \eqref{eq:Ax0 having density zero} implies that
for sufficiently small $r_n$,
\[
  \mu_n(B_n(x,2)\cap A_{x})=\frac{\mu (B(x, 2r_n) \cap A_{x})}{\mu (B(x, r_n))}
  < \frac{1}{C} \epsilon^{Q}.
\]
It follows that for such $n$, the set
$B_n(y,\epsilon)\setminus A_{x}$ has
positive measure and therefore cannot be empty. That is,
$X\setminus A_{x}$ is $\epsilon$-dense in $B_n(x,1)$.

The points in $A_{x}$ can be easily avoided by redefining the approximating
isometries $\phi_n$
such that points~1 and~2 of Definition~\ref{pGH} 
still hold, but with
$3 \epsilon$ rather than $\epsilon$.
\end{proof}

\begin{lemma}\label{lem:Lip-density}
Let $v$ be a Lipschitz function on a metric measure space $(Z,d_Z,\mu_Z)$,
where $\mu_Z$ is doubling. Suppose that
$K\subset Z$ and $z\in Z$ such that
\[
\lim_{r\to 0^+}\frac{\mu_Z(B(z,r)\cap K)}{\mu_Z(B(z,r))}=0.
\]
Then
\[
\text{\rm Lip}\, v(z):=\limsup_{Z\ni y\to z}\frac{|v(z)-v(y)|}{d_Z(y,z)}=\limsup_{Z\setminus K\ni y\to z}\frac{|v(z)-v(y)|}{d_Z(y,z)}.
\]
\end{lemma}

\begin{proof}
Clearly 
\[
\limsup_{Z\ni y\to z}\frac{|v(z)-v(y)|}{d_Z(y,z)}\ge\limsup_{Z\setminus K\ni y\to z}\frac{|v(z)-v(y)|}{d_Z(y,z)}.
\]
Let $y_i\in Z$ be a sequence converging to $z$ such that 
\[
\frac{|v(z)-v(y_i)|}{d_Z(y_i,z)}\to \limsup_{Z\ni y\to z}\frac{|v(z)-v(y)|}{d_Z(y,z)}.
\]
If we have a subsequence of this sequence that lies in $Z\setminus K$, then we have the desired equality.
So suppose without loss of generality that each $y_i\in K$. We claim that for each $\eps>0$ there is some
positive integer $N_\eps$ such that when $i\ge N_\eps$, we have $d_Z(w,y_i)\le \eps d_Z(z,y_i)$
for some $w\in Z\setminus K$. Indeed, if this is
not the case, then there is a positive number $\eps_0$ and a subsequence $i_k$ such that
$B(y_{i_k},\eps_0 d_Z(z,y_{i_k}))\subset K$, in which case by the doubling property of $\mu_Z$ we have
\[
\limsup_{r\to 0^+}\frac{\mu_Z(B(z,r)\cap K)}{\mu_Z(B(z,r))}\ge \frac{1}{C_d^\alpha}>0
\]
where $\alpha$ is the real number for which $2^\alpha \eps_0\ge 4$. This would violate the assumption on
the density of $K$ at $z$. 

Now fixing $\eps>0$, with $w_i\in Z\setminus K$ such that $d_Z(y_i,w_i)\le \eps d_Z(z,y_i)$, we have
\begin{align*}
\frac{|v(z)-v(y_i)|}{d_Z(y_i,z)}\le \frac{|v(z)-v(w_i)|}{d_Z(z,w_i)}\frac{d_Z(z,w_i)}{d_Z(y_i,z)}
   +\frac{|v(w_i)-v(y_i)|}{d_Z(y_i,z)}
   &\le \frac{|v(z)-v(w_i)|}{d_Z(z,w_i)}\frac{d_Z(z,w_i)}{d_Z(y_i,z)}+L\eps\\
   &\le \frac{|v(z)-v(w_i)|}{d_Z(z,w_i)}\left[1+\eps\right]+L\eps
 \end{align*}
 where $L$ is a Lipschitz constant of $v$. Letting $i\to\infty$ followed by $\eps\to 0^+$ gives the desired
 identity.
\end{proof}

%

Now we wish to speak about functions $u_\8$ that are limits of $u_n$ according
to~\eqref{zoomlimit}, with the functions $\phi_n$ adapted to $u$.
Note that the values of $u_n$ on $A_{x}$ are not tested in~\eqref{zoomlimit} by the maps $\phi_n$.  
So $u_\8$ will also be the limit 
of functions $(\widetilde{u})_n$, where $\tilde{u}$ is any McShane extension of 
$u\vert_{B(x,R)\setminus A_{x}}$.
By H\"older's inequality, the $1$-Poincar\'{e} inequality
implies the $p$-Poincar\'{e} inequality for all $1 < p < \infty$.
By~\cite[Proposition 4.3]{KKST}, for each $k\in\N$ there is a \emph{Lipschitz} function $v_k\in BV(X)$
such that 
\[
\mu(\{y\in X\, :\, u(y)\ne v_k(y)\})<1/k.
\]
Since for any measurable set $K\subset X$ we have that the \emph{upper} density of $K$ at almost every
point in $X\setminus K$ is zero, by modifying the set $K_k=\{y\in X\, :\, u(y)\ne v_k(y)\}$ on a set of measure zero
we can assume that $\mu(K_k)<1/k$ and that for every $x\in X\setminus K_k$ we have
\[
\limsup_{r\to 0^+}\frac{\mu(B(x,r)\cap K_k)}{\mu(B(x,r))}=0
\]
and that $v_k$ is asymptotically generalized linear in the sense of~\cite[Theorem~3.7]{Che}
and that
the analysis of Proposition~\ref{prop:Lipschitz continuity around a point} holds for $x$.  By
further enlarging $K_k$ if necessary (without increasing its measure), we can also assume that
$x$ is a Lebesgue point of $\Lip v_k$. Since both $K_k$ and $A_x$ have upper density zero at $x$,
from Lemma~\ref{lem:Lip-density} we know that $x$ is a Lebesgue point also for 
$\text{Lip}\, \widetilde{u}$ of 
any Lipschitz extension $\widetilde{u}$
of $u\vert_{B(x,R)\setminus [A_x\cup K_k]}$ to $B(x,R)$
(of course, this extension could depend on $x$ as well, but in this case we can choose $v_k$ itself to be that
extension).
Note that we then have by Lemma~\ref{lem:Lip-density} that
\[
\text{Lip}\, u\vert_{B(x,R)\setminus A_x}(x)=\text{Lip}\, \widetilde{u}(x)=\text{Lip}\, v_k(x).
\]
Thus the theory developed in~\cite[Section~10]{Che}
is applicable for $x\in X\setminus K_k$. Let $N=\bigcap_{k\in\N}K_k$. Then $\mu(N)=0$, and for
each $x\in X\setminus N$ the theory developed in~\cite[Section~10]{Che} is applicable. Therefore
by~\cite[Section~10]{Che},  
this immediately implies the equality of upper and lower Lipschitz constants for $\widetilde{u}$ and that
\[ 
\Lip u_\8 \equiv \Lip \widetilde{u}(x)=\Lip u\vert_{B(x,R)\setminus A_x}(x),
\]
and this constant {minimal $p$-weak upper gradient, for any $1<p<\infty$, of}
$u_\infty$ is bounded above by $C g(x)$ thanks to
Proposition~\ref{prop:Lipschitz continuity around a point}.
On the other hand, by the lower semicontinuity of BV energy, we know that
$dV(v_k,\cdot)\le \Lip v_k\, d\mu$, and so the Radon-Nikodym derivative of 
$V(v_k,\cdot)$ with respect to $\mu$, which by
Lemma~\ref{lem:equal-energy} is also equal to the Radon-Nikodym derivative $g$ of
$V(u,\cdot)$ with respect to $\mu$ in $X\setminus K_k$, is bounded above by $\Lip v_k$.
Therefore we have
\[
g(x)\le \Lip u_\8\le C\, g(x). 
\]

We collect these observations below.

\begin{theorem}\label{thm:main1}
Let $u \in BV(X)$.  Then for $\mu$-a.e. $x\in X$ and any tangent space $(X_{\infty},d_{\infty},x_{\infty},\mu_{\infty})$, 
any function $u_{\infty}$ that arises as a
limit adapted to $u$ at $x$ has a constant minimal $p$-weak upper gradient for each $p>1$
and that constant is less than $C g(x)$,
where $C$ is as in Proposition \ref{prop:Lipschitz continuity around a point}. Furthermore, 
with $h$ the minimal $1$-weak upper gradient of $u_\8$, we have that $L/(4C_0)\le h\le L$ where $L$ is the constant 
minimal $p$-weak upper gradient, and $C_0$ depends solely on the doubling and the $1$-Poincar\'e constants of $X_\8$.
\end{theorem}

\begin{proof}
The proof of the first part of the theorem follows from the discussions above. Thus it now suffices
to prove the last statement of the theorem. From a telescoping argument for the Lipschitz function $u_\8$
on $X_\8$, we see that whenever $\eps>0$, for $z,w\in X_\8$ with $d(z,w)<\eps$ we have
\[
 |u(z)-u(w)|\le C_0\, d_{X_\8}(z,w)[\mathcal M_{4\eps}h(z)+\mathcal M_{4\eps}h(w)],
\]
where $\mathcal M_rh(o):=\sup_{0<\rho\le r}\vint_{B(o,\rho)}h\, d\mu_\8$ for $o\in X_\8$. Thus it follows
from the local version of~\cite[Theorem~10.2.8]{HKST} that $4C_0\, \mathcal M_{4\eps}h$ is an upper gradient of $u$.
Therefore by the minimality of the constant function $L$ as a $p$-weak upper gradient of $u_\8$, we see that
$L\le 4C_0\, \mathcal M_{4\eps}h$ for each $\eps>0$. Letting $\eps\to 0$ and invoking the Lebesgue differentiation theorem,
we see that $L\le 4C_0 h$. Finally, as $u_\8$ is Lipschitz, the constant function $L$ is also equal to 
$\Lip u_\8$ which is also an upper gradient of $u_\8$, and so by the minimality of $h$ as a $1$-weak upper gradient, 
we see that $h\le L$, completing the proof.
\end{proof}



\begin{theorem} \label{thm:tangentharm}
Let $u\in BV(X)$. Then for $\mu$-a.e. $x\in X$ for which $g(x)>0$,
and for every tangent space
$(X_{\infty},d_{\infty},x_{\infty},\mu_{\infty})$,
any function $u_\infty$ that arises as a limit adapted to $u$ at $x$ satisfies
$u_{\infty}(x_{\infty})=0$ and 
\[
\frac{g(x)}{C}\le \underset{B(y,s)}{\var}u_{\infty}\le \LIP u_{\infty}\le Cg(x)
\]
for every $y\in X_{\infty}$ and $s>0$. 
Furthermore, $u_\infty$ is a function of least gradient.
For $\mu$-a.e.~$x\in X$ for which $g(x)=0$, $u_\infty$ is a constant function.
\end{theorem}

\begin{proof}
The inequality $\underset{B(y ,s)}{\var}u_{\infty}\le \LIP u_{\infty}$ is true by definition, and the inequality $\LIP u_{\infty}\le Cg(x)$ 
follows from Proposition \ref{prop:Lipschitz continuity around a point} and~\cite[Section~10]{Che}.   For the inequality 
$\frac{g(x)}{C}\le \underset{(y,s)}{\var}u_{\infty}$ we first note that by~\cite[Theorem~6.2.1]{Kei04},
\[ 
   \lip \widetilde{u}(x) \leq \underset{B(y,s)}{\var}u_{\infty}, 
\]
where $\widetilde{u}$ is a McShane  extension of $u\vert_{B(x,1)\setminus A_{x}}$ to $B(x,1)$.
Now note by Lemma~\ref{lem:equal-energy} that
\[ 
\frac{g(x)}{C} \leq \lip \widetilde{u}(x).
 \]

By~\cite{Che}
we know that $u_\infty$ is $p$-harmonic
for each $p>1$. Letting $p\to 1^+$, it follows from~\cite[Theorem 3.3]{KLLS}
that $u_\infty$ is a function of least gradient in $X_\infty$.

Finally, if $g(x)=0$ and $x$ is a point of density $1$ for the set $\{y\in X\, :\, g(y)=0\}$, then we can choose for each $n\in\N$
a set $B(x,r_n)\setminus A_{x,1/n}$ as in Proposition~\ref{prop:Lipschitz continuity around a point} such that 
$u$ is $1/n$-Lipschitz on $B(x,r_n)\setminus A_{x,1/n}$.  Thus the limit function $u_\infty$ is $1/n$-Lipschitz continuous
for each $n\in\N$, and so is $0$-Lipschitz, that is, $u_\infty$ is constant.
\end{proof}

The focus of the next section will be to study asymptotic behavior of the characteristic function $\chi_E$
of a set $E$ of finite perimeter at a boundary point.
In considering such behavior, it is not possible to obtain a fruitful 
notion of the asymptotic limit of $\chi_E$ in a manner analogous to the above. Instead of considering a sequence of
scaled versions of $\chi_E$, as with the scaled versions $u_n=[u-u(x)]/r_n$ above, we consider the scaled versions
of the \emph{measures} $\mu_E$ given by $d\mu_{E,n}:=\mu(B(x,r_n))^{-1}\chi_E\, d\mu$, and study weak* limits of 
such measures. The rest of this section discusses how the two notions, one dealing with a scaled version of the function
and the other with a scaled version of the measure, are related.

We fix a sequence $X_n=(X, d_n, x,\mu_n)$ that converges in the pointed measured Gromov-Hausdorff sense
to a tangent space $X_\infty=(X_\infty,d_\infty,x_\infty,\mu_\infty)$ as discussed above, and for such a sequence we
let $\nu_n$ be the measure on $X_n$ given by 
\[
d\nu_n :=\mu(B(x,r_n))^{-1} (u-u(x))/r_n\,  d\mu.
\]
We wish to show that the sequence of measures $\nu_n$ has a
subsequence that converges to the measure $u_\infty \,d\mu_\infty$.

\begin{theorem}\label{thm:main2}
Let $u\in BV(X)$. Then for $\mu$-a.e. $x\in X$ we have the following:
if $(X_{\infty},d_{\infty},x_{\infty},\mu_{\infty})$ is any tangent space
to $X$ at $x$, and $u_{\infty}$ is a function that arises as a
limit adapted to $u$ at $x$, then also
\[
d\nu_n=\mu(B(x,r_n))^{-1}(u-u(x))/r_n\, d\mu \overset{*}{\rightharpoonup} u_\infty \,d\mu_\infty 
\quad\textrm{as }n\to\infty.
\]
\end{theorem}

Naturally, this weak limit is attained along the same subsequence as $u_\infty$ is. In addition to the connection that
the above theorem makes between the way the limit function $u_\infty$ was obtained
above and the tangent-space analysis of sets of finite perimeter in the next section, the theorem also gives an
elegant way of constructing the limit function $u_\8$ \emph{without} having to modify the functions 
$\phi_n$ of Definition~\ref{pGH} to avoid the sets $A_x$.

\begin{proof}
Assume that $x$ is a Lebesgue point of $u$ such that
\[
\lim_{r\to 0}\,\vint_{B(x,r)}|g-g(x)|\,d\mu= 0\quad\textrm{and}\quad \lim_{r\to 0}\frac{V_s(u,B(x,r))}{\mu(B(x,r))}=0,
\]
and such that the conclusion of
Lemma \ref{lem:Lebesgue point lemma} holds.
Let $A$ be the set of all points $y\in X$ such
that for some $0<r\le 2d(x,y)$,
\[
\vint_{B(y,r)}|g-g(x)|\,d\mu+\frac{V_s(u,B(y,r))}{\mu(B(y,r))}>1.
\]
Just as in the proof of
Proposition \ref{prop:Lipschitz continuity around a point}, we get
\begin{equation}\label{eq:densityest}
\lim_{r \to 0} \frac{\mu(A \cap B(x,r))}{\mu(B(x,r))} = 0.
\end{equation}
By Lemma \ref{lem:Lebesgue point lemma} we obtain
\begin{equation}\label{eq:Lebesgue point type condition}
\lim_{r \to 0}\frac{1}{\mu(B(x,r))}\int_{A \cap B(x,r)} \frac{|u-u(x)|}{r}\,d\mu = 0.
\end{equation}

Fix $R>0$ and consider the embeddings 
 $\iota \colon X_\infty \to Z$ and
 $\iota_n \colon X_n \to Z.$ To prove the theorem,
we need to show that whenever $\phi$ is a continuous function supported in $B_{Z}(\iota(x_\8), R)$, we have
\[
\lim_{n\to\infty}\int_{Z}\phi\, \iota_{n,*}(d\nu_n)
=\int_{Z}\phi \, \iota_{*}(u_\infty\,d\mu_\infty).
\]
Just as in Proposition~\ref{prop:Lipschitz continuity around a point},
we have that for
{all sufficiently large}
$n\in\N$, $u|_{B(x,2Rr_n) \setminus A}$
is $C(g(x) + 1)$-Lipschitz. Then, define $\widetilde{u}$ to be the McShane extension of 
$u|_{B(x,2Rr_n) \setminus A}$. Just as in Theorem~\ref{thm:main1}, we see that $u_\8$ is a limit function of
$(\widetilde{u}-\widetilde{u}(x))/r_n$ which are all $C[g(x)+1]$-Lipschitz, 
and the sequence of measures $\mu(B(x,r_n))^{-1}r_n^{-1}(\widetilde{u}-\widetilde{u}(x))\, d\mu$ also converges weakly* to 
$u_\infty \,d\mu_\infty$. See Definition~\ref{def:conv-functions} and the discussion following it.
It thus suffices to show that if
$f_n := (\widetilde{u}-\widetilde{u}(x))/r_n$ 
and $g_n := (u-u(x))/r_n$, then we have that
$\mu(B(x,r_n))^{-1}(f_n-g_n)\, d\mu \overset{*}{\rightharpoonup} 0$.

 Let $\phi$ be a continuous function supported in the ball
 $B_{Z}(\iota(x_\8), R)$.
 Then for all sufficiently large $n$,
\begin{align*}
\bigg\vert \int_{Z} \phi\, \iota_{n,*}\left((f_n-g_n)\, d\mu_n\right) \bigg\vert
&= \bigg\vert\int_{X_n} \phi(\iota_n(y)) (f_n-g_n)(y) \,d\mu_n(y)\bigg\vert \\
&\leq  \frac{\Vert\phi\Vert_{\infty}}{\mu(B(x,r_n))} \int_{B(x,2Rr_n)} \left|\frac{\widetilde{u}-\widetilde{u}(x)}{r_n}-\frac{u-u(x)}{r_n}\right| \,d\mu \\
&\leq  \frac{\Vert\phi\Vert_{\infty}}{\mu(B(x,r_n))} \int_{B(x,2Rr_n) \cap A} 
  \left|\frac{\widetilde{u}-\widetilde{u}(x)}{r_n}\right|+\left|\frac{u-u(x)}{r_n}\right| \,d\mu.
\end{align*}
The last line follows since $\widetilde{u}(x)=u(x)$ and $\widetilde{u}=u$ on $B(x,2Rr_n) \setminus A$. 
The first term converges to zero since
$|\widetilde{u}-\widetilde{u}(x)|/r_n \leq C(g(x)+1)R$ on $B(x,2Rr_n)$ by the 
Lipschitz bound for $\widetilde{u}$ and \eqref{eq:densityest}. Finally, the second term converges to zero by~\eqref{eq:Lebesgue point type condition}.
\end{proof}

\section{Asymptotic limits of sets of finite perimeter}

Let $E\subset X$ be a set of finite perimeter, and
fix a point $x \in \partial^*E$ such that Lemma \ref{lem:perim-growth-bounds} holds.
We will zoom in at $x$ to study the asymptotic properties of $E$.
Let $r_n >0$ with $r_n \to 0$.
In this section, we always consider the sequence
\[ (X_n, d_n, x, \mu_n) := \left(X, \frac{1}{r_n} \cdot d, x, \frac{1}{\mu(B(x, r_n))}\cdot \mu \right) \]
under pointed measured Gromov-Hausdorff convergence.
We also wish to study the behavior of the measure $P(E,\cdot)$ as we zoom in, so let
\[ (X_n, d_n, x, P_n(E, \cdot)) := \left( X, \frac{1}{r_n} \cdot d, x, \frac{r_n}{\mu (B(x, r_n))} P(E, \cdot) \right). \]
Taking subsequences as necessary (not relabeled), we find
the following measures on the limit space $(X_{\infty}, d_{\infty}, x_{\infty})$:
\[ 
    \mu_n \overset{*}{\rightharpoonup} \mu_{\infty}, 
\]
\[ 
    \mu_n(\cdot \cap E) \overset{*}{\rightharpoonup} \mu_{\infty}^E, 
\]
\[ 
    \mu_n(\cdot \cap E^c) \overset{*}{\rightharpoonup} \mu_{\infty}^{E^c}, 
\]
\[ 
    \mbox{and } P_n(E ,\cdot)\overset{*}{\rightharpoonup} \pi_\infty. 
\]
Here,  $P_n$ is the scaled perimeter measure given by
\[
P_n(E, B(z,\rho))=r_n\, \frac{P(E,B(z,r_n\rho))}{\mu(B(x,r_n))}.
\]
In Section \ref{sec:PMGHC} it was noted that a tangent space
$(X_{\infty}, d_{\infty}, x, \mu_{\infty})$ always exists, is geodesic,
and $\mu_{\infty}$ is doubling and supports a $1$-Poincar\'e inequality.
Note that
$\mu_n(\cdot\cap E)+\mu_n(\cdot\cap E^c)=\mu_n$,
and so $\mu_{\infty}^E+\mu_{\infty}^{E^c}=\mu_{\infty}$.
The existence of $\pi_{\infty}$
follows from Lemma~\ref{lem:perim-growth-bounds}: 
by~\eqref{eq:density of perimeter}, for every $k\in\N$,
\begin{equation}\label{eq:bdd-perim-scaled}
\limsup_{n\to\infty}P_n(E,B_n(x,k))=\limsup_{n\to\infty}
\frac{r_n}{\mu (B(x, r_n))} P(E,B(x,k r_n))
\le C_d^{1+\lceil \log_2 k\rceil}.
\end{equation}
At various points in this
section, we specify additional conditions on $x\in \partial^*E$
by excluding $\mathcal H$-negligible parts of $\partial^* E$.

Since
$X$ is geodesic and $\mu$ is doubling, the space satisfies
the following \emph{annular decay property}: there exists
$\delta=\delta(C_d)\in (0,1]$ such that for all $y\in X$, $r>0$,
and $0<\eps<1$, we have
\begin{equation}\label{eq:annular decay property}
\mu(B(y,r)\setminus B(y,r(1-\eps)))\le C \eps^{\delta}\mu(B(y,r)),
\end{equation}
see \cite[Corollary 2.2]{Buc}.
In particular, this property implies that all spheres
have zero $\mu$-measure.
 
We now define two sets in $X_\8$.  Let $(E)_{\infty}$ be the collection
of all points $z\in X_{\infty}$ for which
\[ 
\lim_{r \to 0} \frac{\mu_{\infty}^E (B_{X_\infty}(z, r))}{\mu_{\infty} (B_{X_\infty}(z,r))} = 1,
\]
and let $(E^c)_{\infty}$ be the analogous collection of all points $z\in X_{\infty}$ for which
\[ 
\lim_{r \to 0} \frac{\mu_{\infty}^{E^c} (B_{X_\infty}(z, r))}{\mu_{\infty} (B_{X_\infty}(z,r))} = 1.
\]

\begin{lemma}\label{lem:consequences of annular decay}
	If $z\in X_{\infty}$ and $z_n\in X_n$ with $z_n\to z$ (or, more precisely, $ \iota_n(z_n) \to \iota(z_\8)$ in $Z$),
	then for every $r>0$,
	\[
	\mu_{\infty} (B_{X_{\infty}}(z,r)) = \lim_{n \to \infty} 
	\mu_n(B_n(z_n,r)).
	\]
	The analogous result holds for the measures
	$\mu_n(\cdot \cap E)$ and $\mu_{\infty}^E$, and for
	$\mu_n(\cdot \cap E^c)$ and $\mu_{\infty}^{E^c}$.
\end{lemma}
\begin{proof}
	We prove the result for the measures
	$\mu_n(\cdot \cap E)$ and $\mu_{\infty}^E$; the proofs for the
	other two pairs are analogous.
	Fix $\eta>0$.
	By the lower semicontinuity of measure in open sets under weak* convergence (see e.g. \cite[Proposition 1.62]{AFP}),
	\[
	\mu_{\infty}^E (B_{X_{\infty}} (z,r-\eta))
	=\mu_{\infty}^E (B_{Z} (z,r-\eta))
	\leq \liminf_{n \to \infty} \left[ \iota_{n,*} \mu_n(E \cap \cdot) \right] (B_{Z} (z, r-\eta)). 
	\]
	Here $B_{Z} (z, r-\eta)$ is the ball in $Z$ whose center is the image of $z$ under the isometric embedding $\iota$.
	Letting $\epsilon_n:=d_{Z}(z_n,z)$,
	we have $\epsilon_n \to 0$ and for large $n$,
	\[ 
	B_{Z} (z, r-\eta) \subset B_{Z}(z_n, r-\eta+\epsilon_n)\subset B_{Z}(z_n, r),
	\]
	where again we  label the image of $z_n$ under the
	isometric embedding $\iota_n$ also by $z_n$. Now we can conclude
	\[
	\mu_{\infty}^E (B_{X_{\infty}} (z,r-\eta)) \leq
	\liminf_{n \to \infty} \mu_n(B_n (z_n, r) \cap E). 
	\]
	Thus letting $\eta\to 0$,
	\begin{equation}\label{eq:liminf for muEinfty}
	\mu_{\infty}^E (B_{X_{\infty}} (z,r)) \leq
	\liminf_{n \to \infty} \mu_n(B_n (z_n, r) \cap E). 
	\end{equation}
	
	Again fix $\eta>0$.
	By upper semicontinuity of measure in compact sets under weak* convergence,
	\[
	\mu_{\infty}^E (B_{X_{\infty}}(z,r+2\eta)) \geq 
	\mu_{\infty}^E (\overline{B_{X_{\infty}}(z,r+\eta)}) \geq 
	\limsup_{n \to \infty} \left[ \iota_{n,*} \mu_n(E \cap \cdot) \right] ( \overline{B_{Z}(z, r+\eta)} ).
	\]
 	Again for large $n$,
	\[ \overline{B_{Z}(z, r + \eta)} ) \supset B_{Z} (z_n, r+\eta - \epsilon_n)
	\supset B_{Z} (z_n, r), \]
	and we conclude
	\[
	\mu_{\infty}^E (B_{X_{\infty}}(z,r+2\eta)) \geq
	\limsup_{n \to \infty} \mu_n(B_n (z_n, r) \cap E).
	\]
	Note that the measure $\mu_{\infty}$ also satisfies the annular
	decay property, and that $\mu^E_{\infty}\le \mu_{\infty}$, so spheres in $X_\8$ do not carry
	positive $\mu^E_\8$-weight; therefore letting $\eta\to 0$, we get
	\[
	\mu_{\infty}^E (B_{X_{\infty}}(z,r)) \geq
	\limsup_{n \to \infty} \mu_n(B_n (z_n, r) \cap E).
	\]
	Combining this with \eqref{eq:liminf for muEinfty}
	completes the proof.
\end{proof}

\begin{proposition}  \label{thm:Einfinitydisjoint}
The sets $(E)_{\infty}$ and $(E^c)_{\infty}$ are disjoint.
Moreover, \[\mu_\infty(X_\infty\setminus [(E)_{\infty}\cup(E^c)_{\infty}])=0.\]
\end{proposition}

\begin{proof}
Suppose $z \in (E)_{\infty} \cap (E^c)_\infty$.  Then
\[ 
\lim_{r \to 0} \frac{\mu_{\infty}^E (B_{X_{\infty}}(z, r))}{\mu_{\infty}(B_{X_{\infty}}(z,r))} 
= 1=\lim_{r \to 0} \frac{\mu_{\infty}^{E^c} (B_{X_{\infty}}(z, r))}{\mu_{\infty}(B_{X_{\infty}}(z,r))}.
\]
By the Gromov-Hausdorff convergence, there is a sequence $x_n \in X_n$ 
with $x_n\to z$.
Thus for any small enough $r>0$,
Lemma \ref{lem:consequences of annular decay} gives
\[
2/3 < \frac{\mu_{\infty}^E(B_{X_{\infty}}(z,r))}{\mu_{\infty}(B_{X_{\infty}}(z,r))} 
= \lim_{n \to \infty} \frac{ \mu_n(B_n(x_n, r) \cap E)}{\mu_n( B_n(x_n, r))}
=\lim_{n \to \infty} \frac{ \mu(B(x_n, r r_n) \cap E)}{\mu( B(x_n, r r_n))}
\]
and similarly
\[
2/3 < \frac{\mu_{\infty}^{E^c}(B_{X_{\infty}}(z,r))}{\mu_{\infty}(B_{X_{\infty}}(z,r))} 
= \lim_{n \to \infty} \frac{ \mu(B(x_n, r r_n) \cap E^c)}{\mu( B(x_n, r r_n))}. 
\]
Adding together, we find for all small enough $r>0$ and large enough $n$
\[ 
4/3 < \frac{ \mu(B(x_n, r r_n))}{\mu(B(x_n, r r_n))},
\]
which is not possible. Therefore
$(E)_\infty$ and $(E^c)_\infty$ are disjoint. 

Next, we show that $\mu_\infty(X_\infty\setminus[(E)_\infty\cup(E^c)_\infty])=0$. To this end, we will show that
the Radon-Nikodym derivative of $\mu_\infty^E$ with respect to $\mu_\infty$ is $\mu_\infty$-a.e.~either $1$ or $0$. Let this
Radon-Nikodym derivative be denoted by $\pip$.
Let $A_0:=\{z\in X_\infty :\, 0<\pip(z)<1\}$, and suppose that $\mu_\infty(A_0)>0$.
Then there is some $R>0$ and $0<\delta<1$ such that the set
\[
A:=\{z\in B_{X_\infty}(x_\infty, R) :\, \delta<\pip(z)
<1-\delta\text{ and }z\text{ is a Lebesgue point of }\pip\}
\]
satisfies
\[
\mu_\infty(A)> \delta \mu_\infty(B_{X_\infty}(x_\infty,R)).
\]
We fix $0<\eps<R$ and consider the family of balls $B_{X_\infty}(z,\rho)$, $z\in A$ and $0<\rho<\eps$, such that
\begin{equation}\label{eq:choice of balls covering A}
\delta<\frac{1}{\mu_\infty(B_{X_\infty}(z,\rho))}\int_{B_{X_\infty}(z,\rho)}\pip\, d\mu_\infty<1-\delta.
\end{equation}
As every $z\in A$ is a Lebesgue point of $\pip$, the
corresponding family of closed balls is a
fine cover of $A$, and hence there is a 
pairwise disjoint subfamily $\{\overline{B_i}\}_{i=1}^{\infty}$
such that $\mu_\infty(A\setminus\bigcup_{i=1}^{\infty} \overline{B_i})=0$
and then also $\mu_\infty(A\setminus\bigcup_{i=1}^{\infty} B_i)=0$,
since spheres have $\mu_{\infty}$-measure zero.
Now, observe that
\[
\delta \mu_\infty(B_{X_\infty}(x_\infty,R))< \mu_\infty(A)\le \sum_{i=1}^\infty\mu_\infty(B_i),
\]
and so we can find $N\in\N$ such that 
\begin{equation}\label{eq:cover-measurebdd}
\delta \mu_\infty(B_{X_\infty}(x_\infty,R))< \sum_{i=1}^N\mu_\infty(B_i).
\end{equation}
Denote the center of each $B_i$ by $x^i$.
By Lemma \ref{lem:consequences of annular decay}
we can find $j_\eps\in\N$ such that
whenever $n\ge j_\eps$, there are points $x_n^1,\cdots, x_n^N\in X=X_n$ converging to $x^1, \dots, x^n$
respectively, and a number $\eta>0$ such that
\[
\mu_n(B_n(x_n^i,\rad B_i))
\le (1+\delta^2)\mu_\infty(B_i).
\]
for all $i=1,\ldots,N$.
This gives the second inequality of \eqref{eq:covering} below;
analogously we obtain for all $i=1,\ldots,N$
\begin{align}
(1-\delta^2)\mu_\infty(B_{X_\infty}(x_\infty,R))
  &\le \mu_n(B_n(x,R))\le (1+\delta^2)\mu_\infty(B_{X_\infty}(x_\infty,R)),\label{eq:orignBall}\\
(1-\delta^2)\mu_\infty(B_i)&\le \mu_n(B_n(x_n^i,\rad B_i))\le (1+\delta^2)\mu_\infty(B_i),\label{eq:covering}\\
\delta \frac{1-\delta^2}{1+\delta^2}\le \frac{1-\delta^2}{1+\delta^2}\frac{1}{\mu_\infty(B_i)}\int_{B_i}\pip\, d\mu_\infty
  &\le \frac{\mu_n(B_n(x_n^i,\rad B_i)\cap E)}{\mu_n(B_n(x_n^i,\rad B_i))}\notag\\
    &\qquad\qquad \le \frac{1+\delta^2}{1-\delta^2} \frac{1}{\mu_\infty(B_i)}\int_{B_i}\pip\, d\mu_\infty
     \le \frac{1+\delta^2}{1+\delta}\le 1-\delta/2;\label{eq:density}
\end{align}
here, to obtain~\eqref{eq:density} we also used~\eqref{eq:choice of balls covering A}.
We can also ensure that the collection of balls (``lifts" of $B_i$ to $X_n$)
$\{B(x_n^i,r_n\rad B_i)\}_{i=1}^N$ are pairwise 
disjoint.
Inequality~\eqref{eq:density} tells us that if $\delta>0$ was chosen
small enough, then
\[
\delta/2\le \frac{\mu(B(x_n^i,r_n\rad B_i)\cap E)}
{\mu(B(x_n^i,r_n\rad B_i))}
\le 1-\delta/2.
\]
Now, applying the relative isoperimetric inequality
\eqref{eq:relative isoperimetric inequality} to these balls gives
\[
\delta/2\le 2C_P \frac{r_n\rad B_i}{\mu(B(x_n^i,r_n\rad B_i))}
P(E, B(x_n^i,r_n\rad B_i)).
\]
Thus we obtain, recalling that $\rad B_i<\eps<R$,
\[
\delta\sum_{i=1}^N\mu(B(x_n^i,r_n\rad B_i))
   \le 4C_P\eps r_n \sum_{i=1}^N P(E, B(x_n^i,r_n\rad B_i))
  \le 4C_P\eps r_n P(E, B(x,2r_n R)).
\]
By~\eqref{eq:cover-measurebdd} and~\eqref{eq:covering}, we now have
\begin{align*}
\delta^2(1-\delta^2)\mu_\infty(B_{X_\infty}(x_\infty,R))
&\le \delta(1-\delta^2)\sum_{i=1}^N\mu_\infty(B_i)\\
&\le \delta \sum_{i=1}^N \frac{\mu(B_n(x_n^i,r_n\rad B_i))}{\mu(B(x,r_n))}\\
&\le \frac{4C_P\eps r_n}{\mu(B(x,r_n))} P(E, B(x,2r_n R)).
\end{align*}
Applying~\eqref{eq:orignBall} now gives
\[
\frac{\delta^2(1-\delta^2)}{1+\delta^2} \frac{\mu(B(x,r_n R))}{\mu(B(x,r_n))}
\le \frac{4C_P \eps r_n}{\mu(B(x,r_n))} P(E, B(x,2r_n R)).
\]
By the doubling property of $\mu$ we obtain
\[
0<\frac{\delta^2(1-\delta^2)}{1+\delta^2} \frac{1}{4C_P C_d \eps}\le \frac{r_n}{\mu(B(x,2r_n R))}\, P(E, B(x,2r_n R)).
\]
Now letting $n\to\infty$, by \eqref{eq:density of perimeter} we get
\[
0<\frac{\delta^2(1-\delta^2)}{1+\delta^2} \frac{1}{4C_P C_d \eps}\le C_d
\]
for every $0<\eps<1$, which is not possible. Thus $\mu(A_0)=0$.

Now the claim
$\mu_\infty(X_\infty\setminus[(E)_\infty\cup(E^c)_\infty])=0$
follows from the fact that $\mu_\infty^E+\mu_\infty^{E^c}=\mu_\infty$.
\end{proof}

Note that by the above proposition and by the
Radon-Nikodym theorem, we now have
\[
\mu_\infty^E(A)=\mu_\infty((E)_\infty\cap A)
\quad\textrm{and}\quad \mu_\infty^{E^c}(A)=\mu_\infty((E^c)_\infty\cap A)
\]
for Borel sets $A\subset X_\8$,
and moreover $\partial^*(E)_\infty=X_{\infty}\setminus ((E)_\infty\cup(E^c)_\infty)$.

Now we wish to study the support of the asymptotic perimeter measure
$\pi_{\infty}$ in $(X_{\infty}, d_{\infty}, x_{\infty})$.
We first prove a proposition that 
states that if a point $z \in X_{\infty}$ is in the support of $\pi_{\infty}$,
then it can be seen as the limit of special 
points in $\partial^*E$. 
Recall from \eqref{eq:def of Sigma gamma} that
\[ 
\Sigma_\gamma E:= \left\{ y \in X:\, \liminf_{r \to 0} \min \left\{ \frac{\mu (B(y,r)\cap E)}{\mu(B(y,r))} , 
\frac{\mu(B(y,r)\cap E^c)}{\mu(B(y,r))} \right\} \geq \gamma \right\} \subset \partial^* E 
\]
for some $\gamma=\gamma(C_d,C_P)\in (0,1/2]$.
By \eqref{eq:def of theta} and \eqref{eq:mthbdry minus Sigmagamma} we know that $P(E,\cdot)$ is concentrated on $\Sigma_\gamma E$.

For each $m\in\N$, let
\[ 
G_m :=\left\{ z \in \Sigma_\gamma E:\,   \frac{\gamma}{2 C_P}\le 
      r\frac{ P(E,B(z,r))}{\mu(B(z,r))} \le 2C_d  \quad\text{for all }0<r<\frac1m\right\} .
\]
By Lemma~\ref{lem:perim-growth-bounds} we know that 
\[
  \mathcal{H}\left(\partial^*E\setminus \bigcup_{m\in\N}G_m\right)=0,
\]
and $G_m\subset G_{m+1}$ for all $m\in\N$.
Note that for every $r>0$ the map $z\mapsto P(E, B(z,r))$ is lower semicontinuous,
and so $G_m$ is a Borel set.
Combining the definitions of $\Sigma_{\gamma}E$ and $G_m$,
for every $z\in \Sigma_{\gamma}E\cap G_m$ we find $r_z>0$ such that
\begin{equation}\label{perimetercomp}
r P(E,B(z,r))\le K \min\{\mu(B(z,r)\cap E), \mu(B(z,r)\cap E^c)\}
\end{equation}
for all $0< r \le r_z$, where $K=K(C_d,C_P)$.
Hence we can refine $G_m$ further by considering the set
\[
G_m^*:=\left\{z\in G_m :\, \text{\eqref{perimetercomp} holds for all }0<r<\frac{1}{m}\right\}.
\]
Note then that $G_m^*$ is a Borel set, and that 
$G_m^*\subset G_{m+1}^*$ for $m\in\N$ with 
\[
\mathcal{H}\left(\partial^*E\setminus \bigcup_{m\in\N}G_m^*\right)=0.
\]
%
As $P(E,\cdot)$ is asymptotically doubling by \eqref{eq:density of perimeter},
we know that 
the Lebesgue differentiation theorem holds for the measure $P(E,\cdot)$. Hence 
for any fixed $m\in\N$, 
by the Lebesgue differentiation theorem, $G_m^*$ is of density $1$ (with respect to the measure $P(E,\cdot)$) at
$P(E,\cdot)$-a.e. $x\in G_m^*$. 
It is at such a point that we will zoom in and take our limiting measures.

\begin{proposition} \label{prop:perimeterinfinitylb}
Let $(X_n, d_n, x, \mu_n)$ be a pointed measured Gromov-Hausdorff convergent sequence
such that the base point $x$
is a point of $P(E,\cdot)$-density $1$ for $G_m^*$ for some $m\in\N$.
Suppose that $z \in X_\infty$ is such that
\[ \pi_\infty ( B_{X_\infty}(z,R) ) > 0 \]
for all $R>0$. Then there is a sequence
$z_n \in X_n$ that converges
to $z$ (in $Z$) such that each $z_n$ is in $G_m^*$.
\end{proposition}

\begin{proof}
Fix $R>0$.
By the Gromov-Hausdorff convergence there exist sequences
$\epsilon_n\to 0$ and $x_n\in X_n=X$ such that
$d_{Z}(x_n,z)<\epsilon_n$, and then
by lower semicontinuity under weak* convergence,
\begin{equation} \label{assump}
\liminf_{n \to \infty} 
r_n\frac{P(E, B(x_n, (R+\epsilon_n)r_n))}{\mu(B(x, r_n))}
\ge \liminf_{n \to \infty} 
r_n\frac{P(E, \iota_n^{-1}(B_{Z}(z, R)))}{\mu(B(x, r_n))} \ge \pi_\infty (B_{X_\infty}(z,R) )>0.
\end{equation}
We would like to
know that there exists $\widetilde{x}_n\in G_m^* \cap B(x_n, (R+\epsilon_n)r_n)$
for all sufficiently large $n\in\N$.  Suppose that this is 
not the case.  Then there is a subsequence $n_k$ such that
$B(x_{n_k}, (R+\epsilon_{n_k})r_{n_k})$ is disjoint from $G_m^*$.  
Choose $M>0$ large enough so that for all $k\in\N$,
\[ B(x_{n_k}, (R+\epsilon_{n_k})r_{n_k}) \subset B(x, M r_{n_k}). \]
As $B(x_{n_k}, (R+\epsilon_{n_k})r_{n_k})$ and $B(x, M r_{n_k}) \cap G_m^*$ are disjoint, 
we have
\[ 
P(E, B(x_{n_k}, (R+\epsilon_{n_k})r_{n_k})) + P(E, B(x, M r_{n_k}) \cap G_m^*)  \leq P(E, B(x, M r_{n_k})),
\]
which is equivalent to
\[
r_{n_k}\frac{  P(E, B(x_{n_k}, (R+\epsilon_{n_k})r_{n_k})) }{\mu(B(x, r_{n_k}))} 
   + r_{n_k}\frac{ P(E, B(x, M r_{n_k}) \cap G_m^*) }{\mu(B(x, r_{n_k}))}
   \leq r_{n_k}\frac{ P(E, B(x, M r_{n_k})) }{\mu(B(x, r_{n_k}))}.
\]
Call the left-hand side of this inequality $A_k + B_k$, and the right-hand side $C_k$.  
The assumption that $x$ is a point of density $1$ in $G_m^*$ implies that $B_k / C_k \to 1$ as $k \to \infty$.  
Thus $A_k/C_k\to 0$ as $k\to\infty$.
On the other hand, by the definition of $G_m$, we must have $C_k\le 2M^{-1}C_d^{2+\log_2(M)}<\infty$
for large $k\in\N$.
Therefore
$A_k\to 0$ as $k \to \infty$, which contradicts~\eqref{assump}.
Thus, there is some $N_1\in\N$ such that there is a point
$\widetilde{x}_n\in G_m^*\cap B(x_n, (R+\epsilon_n)r_n)$ for 
all $n\ge N_1$.
We rename this sequence $\widetilde{x}_n^1$.  Similarly, there exists 
a sequence
\[ \widetilde{x}_n^2 \in G_m^*\cap B(x_n, (2^{-1}R+\epsilon_n)r_n) \]
for all $n \geq N_2 > N_1$.  We continue inductively in this fashion to find for each $k\in\N$, 
\[ \widetilde{x}_n^k \in G_m^*\cap B(x_n, (2^{-k}R+ \epsilon_n)r_n) \]
for all $n \geq N_k >  N_{k-1}$.  Now
\[
 d_{Z} (\widetilde{x}_n^k, z)  \leq  d_{Z} (\widetilde{x}_n^k, x_n) + d_{Z} (x_n, z)
                                             \leq  2^{-k} R + \epsilon_n + \epsilon_n.
\]                         
For $n \in \left[N_{k}, N_{k+1}\right)$, set $z_n  := \widetilde{x}_n^k$.
Then $z_n$ has the desired properties.
\end{proof}

Note also that the support of $\pi_{\infty}$
is contained in $X_{\infty}$; this can be seen as follows.
If $z\in Z$ such that
$\pi_\infty (B_{Z}(z,R) )>0$ for all $R>0$,
then there exists a sequence $x_n\in X_n$ such that $x_n \to z$ in $Z$.
It follows that $z\in X_\8$.

We now provide growth estimates for the measure $\pi_\infty$. 

\begin{theorem}\label{thm:Main2}
Consider the sequence $(X_n, d_n, x, \mu_n)$
such that  $x$ is
a point of $P(E,\cdot)$-density $1$ for $G_m^*$ for some $m\in\N$.
Suppose that $z \in X_\infty$ is such that
\[ 
\pi_\infty (B_{X_\infty}(z,R)) > 0
\]
for all $R>0$.  Then 
\begin{equation}\label{eq:T1}
\frac{1}{C} \frac{\mu_\infty(B_{X_\infty}(z,r))}{r}
\le \pi_\infty(B_{X_\infty}(z,r))
\le C \frac{\mu_\infty(B_{X_\infty}(z,r))}{r}
\quad\textrm{for all }r>0,
\end{equation}
where $C=C(C_d,C_P)$,
and
\begin{equation}\label{eq:T2}
\pi_\infty((E)_\infty\cup(E^c)_\infty)=0.
\end{equation}
\end{theorem}

\begin{proof}
By Proposition \ref{prop:perimeterinfinitylb}, there is a 
sequence $z_n\in G_m^*$ that converges to $z$ in $Z$.
Fix $r>0$ and $0<\eta<r/2$.
Using the basic properties of weak* convergence as in the proof of
Proposition \ref{prop:perimeterinfinitylb}, we find a sequence of positive
numbers $\epsilon_n$
with $\lim_{n\to\infty}\epsilon_n=0$ such that
\begin{equation}\label{bigineq}
\liminf_{n \to \infty} r_n \frac{P(E, B(z_n, (r + \epsilon_n) r_n))}{\mu (B(x, r_n))} \geq \pi_\infty (B_{X_\infty}(z,r))
 \geq  \limsup_{n \to \infty} r_n \frac{P(E, B(z_n, (r-\eta) r_n))}{\mu (B(x, r_n))}.
\end{equation}
We can rewrite the term on the right-most side of~\eqref{bigineq} as
\[ 
\limsup_{n \to \infty} (r - \eta) r_n \frac{ P(E, B(z_n, (r - \eta) r_n))}{\mu(B(z_n, (r-\eta) r_n))} 
  \cdot \frac{\mu(B(z_n, (r-\eta)r_n))}{(r-\eta) \mu(B(x,r_n))}.
\]
Since $z_n \in G_m^*$, we know that
\[ 
(r - \eta) r_n\frac{P(E, B(z_n, (r - \eta) r_n))}{\mu(B(z_n, (r-\eta) r_n))} 
\geq \frac{\gamma}{2 C_P}
\]
for all large enough $n$ (that is, when $(r-\eta) r_n<1/m$).  
Additionally, by Lemma~\ref{lem:consequences of annular decay},
\begin{align*}
\limsup_{n\to\infty} \frac{\mu(B(z_n, (r-\eta)r_n))}{(r-\eta) \mu(B(x,r_n))}
&= \frac{1}{(r-\eta)}\limsup_{n\to\infty} \frac{\mu(B(z_n, (r-\eta)r_n))}{\mu(B(x,r_n))}\\
  &= \frac{\mu_\infty(B_{X_{\infty}}(z,r-\eta))}{r-\eta}\\
  &\ge \frac{1}{C_d^2}\, \frac{\mu_\infty(B_{X_{\infty}}(z,r))}{r},
\end{align*}
since $\mu_{\infty}$ is doubling with constant $C_d^2$.
Thus by~\eqref{bigineq}, we get
\[
\pi_\infty(B_{X_\infty}(z,r))\ge \frac{\gamma}{2C_d^2 C_P} \frac{\mu_\infty(B_{X_{\infty}}(z,r))}{r}.
\]
Next we rewrite the term on the left-most side of~\eqref{bigineq} as
\[ 
\liminf_{n \to \infty} (r + \epsilon_n) r_n
\frac{ P(E, B(z_n, (r + \epsilon_n) r_n))}{\mu(B(z_n, (r + \epsilon_n) r_n))} 
   \cdot \frac{\mu(B(z_n, (r + \epsilon_n)r_n))}{(r + \epsilon_n) \mu(B(x,r_n))} .
\]
Since $z \in G_m^*$, we know that
\[
 (r + \epsilon_n) r_n\frac{P(E, B(z_n, (r + \epsilon_n) r_n))}{\mu(B(z_n, (r + \epsilon_n) r_n))}\le 2C_d
\]
for all large enough $n$.  
Similarly to above, we obtain
\[
\lim_{n\to\infty}\frac{\mu(B(z_n, (r + \epsilon_n)r_n))}{(r + \epsilon_n) \mu(B(x,r_n))}
=  \frac{\mu_\infty(B_{X_\infty}(z,r))}{r},
\]
whence from~\eqref{bigineq} we obtain
\[
\pi_\infty(B_{X_\infty}(z,r))
\le 2C_d\, \frac{\mu_\infty(B_{X_\infty}(z,r))}{r}.
\]
This proves~\eqref{eq:T1}.

It now only remains to show~\eqref{eq:T2}. 
It suffices to show that when $z\in (E)_\infty\cup (E^c)_\infty$, 
for each $k\in\N$ there is some $r_z>0$ such that 
$\pi_\infty(B_{X_\infty}(z,r_z))
\le C \mu_\infty(B_{X_\infty}(z,r_z))/(kr_z)$, from which
we will know (via~\eqref{eq:T1}) that there must be $\rho_z>0$ with
$\pi_\infty(B_{X_\infty}(z,\rho_z))=0$.
Fix $k\in\N$. Suppose $z\in (E)_\infty$ is in the support of $\pi_{\infty}$.
Then there is some
$r_z>0$ such that
\[
\frac{\mu_\infty(B_{X_\infty}(z,r_z)\cap (E^c)_\infty)}{\mu_\infty(B_{X_\infty}(z,r_z))}<\frac{1}{k}.
\]
Let $\epsilon_n\to 0$ and $G_m^*\ni z_n \to z$ as given by the conclusion of Proposition~\ref{prop:perimeterinfinitylb}.
By Lemma \ref{lem:consequences of annular decay} we have
\begin{align*}
\frac{1}{k}\lim_{n\to\infty}\frac{\mu(B(z_n,r_z r_n))}{\mu(B(x,r_n))}=
\frac{\mu_\infty(B_{X_\infty}(z,r_z))}{k}
&> \mu_\infty(B_{X_\infty}(z,r_z)\cap (E^c)_\infty)\\
 &= \lim_{n\to\infty}\frac{\mu(B(z_n,r_z r_n)\cap E^c)}{\mu(B(x,r_n))}.
\end{align*}
Thus for all large enough $n\in\N$,
\begin{align*}
\frac{1}{k} \mu(B(z_n,r_z r_n))\ge \mu(B(z_n,r_z r_n)\cap E^c).
\end{align*}
Since $z_n\in G_m^*$, by~\eqref{perimetercomp} we have that for all large enough $n\in\N$, 
\[
r_z r_n
\frac{ P(E,B(z_n, r_z r_n))}{\mu(B(z_n,r_z r_n))}
\le K \frac{\mu(B(z_n,r_z r_n)\cap E^c)}{\mu(B(z_n,r_z r_n))}
\le \frac{K}{k}.
\]
Thus for any $\eta>0$,
\begin{align*}
\pi_\infty(B_{X_\infty}(z,r_z-\eta))
\le \liminf_{n\to\infty}
r_n\frac{ P(E,B(z_n,r_z r_n))}{\mu(B(x,r_n))}
  &\le \frac{ K}{k} \frac{1}{r_z}
  \liminf_{n\to\infty}\frac{\mu(B(z_n,r_z r_n))}{\mu(B(x,r_n))}\\
&= \frac{ K}{k} \frac{\mu_\infty(B_{X_\infty}(z,r_z))}{r_z}
\end{align*}
by Lemma~\ref{lem:consequences of annular decay}.
Since $K=K(C_d,C_P)$, letting $\eta\to 0$ gives  
\[
 \pi_\8(B_{X_\8}(z,r_z))\le \frac{K}{k}\, \frac{\mu_\infty(B_{X_\infty}(z,r_z))}{r_z}.
\]
By choosing $k\in\N$ large enough, the above would violate the left-hand inequality of~\eqref{eq:T1}, and so
$z$ cannot be in the support of $\pi_\8$.
Thus, there is some $\rho_z>0$
with $\pi_\infty(B(z,\rho_z))=0$. Since this happens for every 
$z\in (E)_\infty$, we know that $\pi_\infty$ does not charge $(E)_\infty$.
Indeed, with 
\[
(E)_\infty \subset U:=\bigcup_{z\in (E)_\infty}B(z,\rho_z),
\]
an open set containing $(E)_\infty$, we have $\pi_\infty(U)=0$.
A similar argument gives the existence of an open set
$V\supset (E^c)_\infty$
with $\pi_\infty(V)=0$. This completes the proof.
\end{proof}

Next we show that the set $(E)_\infty$ is of locally finite perimeter
in the space $X_{\infty}$. 
Denote by $\mathcal{H}_\infty$ the co-dimension $1$ Hausdorff measure
in the space $(X_{\infty},d_{\infty},\mu_{\infty})$.

\begin{theorem}\label{thm:main3}  
For all $R>0$,
we have $P((E)_{\infty},B_{X_\infty}(x_{\infty},R))<\infty$.
The measures $P((E)_\infty,\cdot)$, $\pi_\infty$, and
$\mathcal{H}_\infty({\partial^*(E)_\infty}\cap \cdot)$ are comparable.
The sets $(E)_\infty$ and $(E^c)_\infty$ are open in $X_\infty$.
\end{theorem}

\begin{proof}
To prove the first claim,
we use a discrete convolution construction.
Assume for simplicity that $R=1$.
Fix $0<\epsilon<1/9$, and 
take a maximal $\epsilon$-separated set
$\{z_k\}_{k=1}^{\infty}\subset X_{\infty}$. Then the balls $B_k:=B_{X_\infty}(z_k,\epsilon)$ cover $X_{\infty}$
and $B_{X_\8}(z_k,14\eps)$ have bounded overlap.
For each $k$ we can find points
$z_{k,n}\in X_n$ converging to $z_k$ (in $Z$).
Thus, by considering a tail-end of the sequence if necessary, 
there is a sequence $0<\delta_n\to 0$ with $\delta_n < \epsilon$
such that $d_{Z}(\iota_\8(z_k),\iota_n(z_{k,n}))<\delta_n$, and by
Lemma~\ref{lem:consequences of annular decay},
\begin{equation}\label{eq:Bk and weak star convergence 1a}
(1-\delta_n)\,\mu_\infty^E(B_{X_\infty}(z_k,\epsilon))
\le \mu_n(B_n(z_{k,n},\epsilon)\cap E)
\end{equation}
and
\begin{equation}\label{eq:Bk and weak star convergence 2a}
\mu_n(B_n(z_{k,n},\epsilon))
\le (1+\delta_n)\mu_\infty(B_{X_\infty}(z_k,\epsilon))
\le (1+\delta_n)^2\, \mu_n(B_n(z_{k,n},\epsilon)).
\end{equation}
Observe that we need only do this for the points
$z_k\in B_{X_\infty}(x_\infty,2)$,
of which there are only finitely many,
and thus we can choose $\delta_n>0$ such that the above hold for all
corresponding indices $k$.
By the bounded overlap property of the balls $B_{X_\8}(z_k,14\epsilon)$, we also have
that for each such positive integer $n$, the collection of balls $B_{X_n}(z_{k,n},6\epsilon)$ has a
bounded overlap; this will be needed in the computations~\eqref{eq:global-dominance}.

Now take a partition of unity by means of $C/\epsilon$-Lipschitz functions
$\phi_k\in \Lip(X_{\infty};[0,1])$ with
$\supp(\phi_k)\subset B_{X_\infty}(z_k,2\epsilon)$ for
each $k\in\N$; see e.g. \cite[p. 104]{HKST}. 
Let $u:=\chi_E$, and for each $n \in\N$ we set
\begin{equation}\label{eq:discrete convolution def}
v_n^\epsilon:=\sum_{k=1}^{\infty}u_{B_n(z_{k,n},\epsilon)}\phi_k,
\end{equation}
where
\[
u_{B_n(z_{k,n},\epsilon)}=\vint_{B_n(z_{k,n},\epsilon)}u\,d\mu_n=\vint_{B(z_{k,n},r_n\epsilon)}u\, d\mu
 =\frac{\mu(B(z_{k,n},r_n\epsilon)\cap E)}{\mu(B(z_{k,n},r_n\epsilon))}.
\]
Let $l\in\N$ such that $B_l\cap B_{X_\infty}(x_\infty,1)\neq \emptyset$.
Given $y_1,y_2\in B_l$, we estimate
\begin{align*}
|v_n^\epsilon(y_1)-v_n^\epsilon(y_2)|
&=\left| \sum_{k=1}^{\infty}u_{B_n(z_{k,n},\epsilon)}\phi_k(y_1)
-\sum_{k=1}^{\infty}u_{B_n(z_{k,n},\epsilon)}\phi_k(y_2)\right|\\
&=\left| \sum_{k=1}^{\infty}(u_{B_n(z_{k,n},\epsilon)}
-u_{B_n(z_{l,n},\epsilon)})(\phi_k(y_1)-\phi_k(y_2))\right|\\
&\le \sum_{k=1}^{\infty}|u_{B_n(z_{k,n},\epsilon)}
-u_{B_n(z_{l,n},\epsilon)}||\phi_k(y_1)-\phi_k(y_2)|\\
&=\sum_{\substack{k\in\N \\ B_{X_\infty}(z_{k},2\epsilon)\cap B_l\neq\emptyset}}|u_{B_n(z_{k,n},\epsilon)}
-u_{B_n(z_{l,n},\epsilon)}||\phi_k(y_1)-\phi_k(y_2)|\\
&\le C\sum_{\substack{k\in\N \\ B_{X_\infty}(z_{k},2\epsilon)\cap B_l\neq\emptyset}}|u_{B_n(z_{k,n},\epsilon)}
-u_{B_n(z_{l,n},\epsilon)}|\frac{d_\infty(y_1,y_2)}{\epsilon}.
\end{align*}
Note that for the indices $k$ in the last sum, we have $d_\infty(z_k,z_l)\le 3\epsilon$ and
so $d_n(z_{k,n},z_{l,n})\le 5\epsilon$.
Thus, each ball $B_n(z_{k,n},\epsilon)$ is contained in $B_n(z_{l,n},6\epsilon)$.
Thus, we can continue the estimate for via the Poincar\'e inequality:
\begin{align*}
|v_n^\epsilon(y_1)-v_n^\epsilon(y_2)|
&\le C\vint_{B_n(z_{l,n},6\epsilon)}
|u-u_{B_n(z_{l,n},6\epsilon)}|\,d\mu_n
\frac{d_\infty(y_1,y_2)}{\epsilon}\\
&\le C \epsilon\, \frac{d_\infty(y_1,y_2)}{\epsilon}\, 
\frac{ P_n(E,B_n(z_{l,n},6\epsilon))}{\mu_n(B_n(z_{l,n},6\epsilon))}.
\end{align*}
Thus, 
we get for $y\in B_l$,
\[
\Lip v_n^\epsilon(y)\le C \frac{ P_n(E,B_n(z_{l,n},6\epsilon))}{\mu_n(B_n(z_{l,n},6\epsilon))}.
\]
Therefore, in $B_{X_\infty}(x_\infty,1)$,
\begin{equation}\label{eq:dominant-measure}
\Lip v_n^\epsilon\le C  \sum_{l=1}^{\infty}\chi_{B_l}
\frac{ P_n(E,B_n(z_{l,n},6\epsilon))}{\mu_n(B_n(z_{l,n},\epsilon))}.
\end{equation}
By the definition of pointed measured Gromov-Hausdorff convergence,
$\lim_{n \to \infty} d_{Z}(\iota_n(x),\iota(x_\8))=0$, see Definition~\ref{def:pmGH} 
and the discussion preceding it. 
If $B_l\cap B_{X_\infty}(x_\infty,1)\neq \emptyset$,
then $B_n(z_{l,n},2\epsilon)\cap B_n(x,1+\eps)\neq \emptyset$. 
Hence by~\eqref{eq:Bk and weak star convergence 2a} 
and by the bounded overlap of the family $B_n(z_{k,n},6\epsilon)$, $k\in\N$, we have 
\begin{equation}\label{eq:global-dominance}
\begin{split}
\int_{B_{X_\infty}(x_{\infty},1)}\Lip v_n^\epsilon\,d\mu_{\infty}
&\le C\sum_{\substack{l\in \N \\
		B_n(z_{l,n},2\epsilon)\cap B_n(x,1+\eps)\neq \emptyset}}\mu_{\infty}(B_l)
\frac{ P_n(E,B_n(z_{l,n},6\epsilon))}{\mu_n(B_n(z_{l,n},\epsilon))}\\
&\le C(1+\delta_n) \sum_{\substack{l\in \N \\ B_n(z_{l,n}, 2\epsilon)\cap B_n(x,1+\eps)\neq \emptyset}}
P_n(E,B_n(z_{l,n},6\epsilon))\\
&\le C  P_n(E,B_{n}(x,1+9\epsilon))\\
&\le C  P_n(E,B_{n}(x,2)).
\end{split}
\end{equation} 
This remains bounded as $n\to\infty$, see~\eqref{eq:bdd-perim-scaled}.
We can do the above for a sequence $\epsilon_i\to 0$, with $n=n(i)\to \infty$ and $\delta_{n_i}\to 0$,
to obtain a sequence of functions
$v_i=v_{n(i)}^{\epsilon_i}\in\Lip(B_{X_{\infty}}(x_{\infty},1))$.
Since
$V(v_i,(B_{X_{\infty}}(x_{\infty},1))$ is bounded
by \eqref{eq:global-dominance},
we find a subsequence, also denoted by $v_i$,
such that $v_i\to w$ in $L^1(B_{X_{\infty}}(x_{\infty},1))$,
see \cite[Theorem 3.7]{Mir}. By lower semicontinuity,
\begin{equation}\label{w is in BV}
V(w,B_{X_{\infty}}(x_{\infty},1))
\le \liminf_{i\to\infty}\int_{B_{X_\infty}(x_{\infty},1)}\Lip v_i\,d\mu_{\infty}
\le C \limsup_{n\to\infty}P_n(B_n(x,2))\le \pi_\infty(x_{\infty},3)
\end{equation}
and so
$w\in BV(B_{X_{\infty}}(x_{\infty},1))$.
We need to check that $w= \chi_{(E)_{\infty}}$ in $L^1(B_{X_{\infty}}(x_{\infty},1))$.
To do so, fix $y\in B_{X_\infty}(x_{\infty},1)\cap (E)_\infty$
and fix $\eta\in (0,1)$. Then by definition of $(E)_\8$, for large enough $i\in\N$ we have
\[
\frac{\mu_{\infty}^E(B_{X_\infty}(y,4\epsilon_i))}
{\mu_{\infty}(B_{X_\infty}(y,4\epsilon_i))}\ge 1-\eta.
\]
We denote the covering of $X_{\infty}$ corresponding to an index $i\in\N$
by $B_k^i:=B(z_k^i,\epsilon_i)$.
It follows that for all balls $B_k^i$
with $2B_k^i$ containing $y$, we have (note that $\mu_{\infty}$
is doubling with constant $C_d^2$)
\[
\frac{\mu_{\infty}^{E^c}(B_{X_\infty}(z_k^i, \epsilon_i))}
{\mu_{\infty}(B_{X_\infty}(z_k^i, \epsilon_i))}
\le C_d^6\frac{\mu_{\infty}^{E^c}(B_{X_\infty}(y,4\epsilon_i))}
{\mu_{\infty}(B_{X_\infty}(y,4\epsilon_i))}
\le C_d^6\eta.
\]
Thus by \eqref{eq:Bk and weak star convergence 1a} and \eqref{eq:Bk and weak star convergence 2a},
\[
\frac{\mu_{n_i}(B_{n_i}(z_{k,n_i},\epsilon_i)\cap E)}{\mu_{n_i}(B_{n_i}(z_{k,n_i},\epsilon_i))}
\ge \frac{1-\delta_{n_i}}{1+\delta_{n_i}}
\frac{\mu_{\infty}^E(B_{X_\infty}(z_k^i, \epsilon_i ))}
{\mu_{\infty}(B_{X_\infty}(z_k^i, \epsilon_i))}
\ge \frac{1-\delta_{n_i}}{1+\delta_{n_i}}(1-C_d^6\eta).
\]
Now, by definition of the discrete
convolutions \eqref{eq:discrete convolution def}, we have
\[ v_i(y) \geq \frac{1-\delta_{n_i}}
{1+\delta_{n_i}} (1-C_d^6 \eta). \]
Letting $i\to\infty$, we get
\[
w(y)\ge 1-C_d^6\eta.
\]
Since $\eta>0$ was arbitrary, we conclude $w(y)=1$ (the values taken on by the functions $v_i$ are between $0$ and $1$, so
necessarily $w(y)\le 1$). Similarly, we get $w(y)=0$ for all $y\in (E^c)_{\infty}$.
Also by Proposition~\ref{thm:Einfinitydisjoint} we know that $\mu_\8(X_\8\setminus ((E)_\8\cup (E^c)_\8))=0$.
Thus $w=\chi_{(E)_{\infty}}$ as functions in
$L^1(B_{X_{\infty}}(x_{\infty},1))$.
Recall that we are assuming $R=1$ just for convenience; we conclude that
$\chi_{(E)_{\infty}}\in BV(B_{X_{\infty}}(x_{\infty},R))$
for all $R>0$.

Next, for $z\in X_\infty$ and $r>0$,
by an argument analogous to
that leading to~\eqref{w is in BV}, we obtain
\begin{equation}\label{eq:P controlled by pi}
P((E)_\infty, B(z,r))\le \pi_\infty(z,3r).
\end{equation}

From the final part of the proof of Theorem~\ref{thm:Main2},
we know that 
$\pi_\infty$ is supported inside $X\setminus (U\cup V)$,
where $U$ and $V$ are
(open) neighborhoods of $(E)_\infty$ and $(E^c)_\infty$, respectively.
Conversely,
if $z\in X_{\infty}\setminus [(E)_\infty\cup(E^c)_\infty]$,
which we recall is the same set as
$\partial^*(E)_\infty$, then $z$ is in the support of 
$P((E)_\infty,\cdot)$ by the relative isoperimetric inequality
\eqref{eq:relative isoperimetric inequality}.
Thus by \eqref{eq:P controlled by pi}, $z$ is in the support of  
$\pi_\infty$.
In conclusion, the support of $\pi_\infty$ is exactly $\partial^*(E)_\infty=X_{\infty}\setminus [U\cup V]$.
Moreover,
by~\eqref{eq:T1} and Lemma~\ref{lem:densities and Hausdorff measures} we 
know that $\pi_\infty$ is comparable to $\mathcal{H}_\infty({\partial^*(E)_\infty}\cap \cdot)$.
By~\eqref{eq:def of theta} we know that
$P((E)_\infty,\cdot)$ is also comparable to
$\mathcal{H}_\infty({\partial^*(E)_\infty}\cap \cdot)$.
Thus the three
measures $\pi_\infty$, $P((E)_\infty,\cdot)$, and
$\mathcal{H}_\infty({\partial^*(E)_\infty}\cap \cdot)$ are all
comparable.
Finally, by the relative isoperimetric inequality
and the fact that $P((E)_\infty,\cdot)$ does not see the set $U$,
it follows that $U$ cannot intersect $(E^c)_\infty$. Thus
$(E)_\infty=U$ and similarly $(E^c)_\infty=V$.
\end{proof}

\begin{remark}
If $\mu$ is an Ahlfors $s$-regular measure for some $s>1$
(recall \eqref{eq:Ahlfors regular measure}), then
it is straightforward to verify that $\mu_{\infty}$ is
also Ahlfors $s$-regular in $X_{\infty}$, and then
by Theorems \ref{thm:Main2} and \ref{thm:main3},
$P((E)_\infty,\cdot)$
(and $\pi_\infty$) are Ahlfors $(s-1)$-regular measures
in $X_{\infty}$.
This corresponds to what we get in a Euclidean space $\R^n$,
for $s=n$.
\end{remark}

\section{Asymptotic quasi-least gradient property}

 From Theorem~\ref{thm:tangentharm} we
now know that asymptotic limits ($\mu$-a.e.) of a BV function outside of the Cantor and jump parts
of the function are of least gradient.
We will show in this section that at co-dimension $1$
almost every point of the measure-theoretic boundary of a set $E$ of finite perimeter,
any limit set $(E)_\8$ 
is a set of quasiminimal boundary surface as defined in~\cite{KKLS}, that is, $\chi_{(E)_\8}$ is of
quasi-least gradient. First we develop some preliminary 
results that are also of independent interest.

\subsection{Asymptotic minimality for sets of finite perimeter}

The following theorem shows that given a set $E$ of finite perimeter, at essentially almost every point in $\bdry^*E$ the set $E$ is 
asymptotically a minimal surface; compare this to~\cite[Proposition~5.7]{A}, where a weaker notion of asymptotic quasiminimality
is established, where the quasiminimality condition requires to compare (locally) the perimeter of $E$ with the perimeter
of modifications of $E$ by balls alone. 

\begin{theorem}\label{Asymptotic-FinitePerimeter}
Let $E \subset X$ be a set of finite perimeter.  Then
\[ 
\lim_{r \to 0} \left( \frac{ \inf_{u \in BV_c(B(x_0, r))} V(\chi_E + u, B(x_0, r)) }{P(E, B(x_0, r))} \right) \geq 1 
\]
for $P(E,\cdot)$-a.e. $x \in X$. 
\end{theorem}

\begin{proof}
Let
\[
A := \left\{ x \in X :\, \liminf_{r \to 0} 
      \left( \frac{ \inf_{u \in BV_c(B(x, r))} V(\chi_E + u, B(x, r)) }{P(E, B(x, r))} \right) < 1 \right\} 
\]
Note that $A$ is the increasing limit of sets $A_n$ where
\[ 
A_n :=  \left\{ x \in X:\, \liminf_{r \to 0} 
    \left( \frac{ \inf_{u \in BV_c(B(x, r))} V(\chi_E + u, B(x, r)) }{P(E, B(x, r))} \right) < 1-\frac{1}{n} \right\},\quad n\in\N.
\]
It therefore suffices to show that each $A_n$ satisfies $P(E,A_n)=0$. To this end,
fix $n\in\N$.
Then for every $x \in A_{n}$, there exist $r_i^x \to 0$ and
$u_i^x \in BV_c(B(x, r_i^x))$ with
\begin{equation} \label{gap}
\frac{V(\chi_E + u_i^x, B(x,r_i^x))}{P(E, B(x,r_i^x))} < 1- n^{-1}.
\end{equation}
Furthermore, as $P(E,X)<\infty$, for every $x\in X$
we have $P(E,\partial B(x,r))=0$ for $\mathcal{H}^1$-almost every $r>0$.
We can therefore choose $r_i^x>0$ such that in addition to the above,
$P(E,\partial B(x,r_i^x))=0$ for every $x\in A_{n}$.
Fix $k\in\N$ such that $1/k < \tfrac 14 \diam X$.
The collection $\{\overline{B}(x,r_i^x)\, :\, 0<r_i^x<1/k\}_{x\in A_{n}}$
is a fine cover of $A_{n}$, that is, for every $x\in A_{n}$
we have $\inf_i r_i^x=0$. By~\eqref{eq:density of perimeter} we know that
$P(E,\cdot)$
is asymptotically doubling, and so it
satisfies the Vitali covering theorem, see~\cite[Theorem~3.4.3]{HKST}.
So we can pick a countable pairwise disjoint collection 
$\{B^k_j=B(x_j^k,r_j^k)\}_{j=1}^{\infty}=:\mathcal{G}_k$ such that, recalling also that $P(E,\partial B)=0$ for each
$B\in \mathcal{G}_k$,
\begin{equation}\label{eq:Vital}
P\left(E,A_{n}\setminus \bigcup_{B \in \mathcal{G}_k} B \right)
=P\left(E,A_{n}\setminus \bigcup_{B \in \mathcal{G}_k} \overline{B} \right) = 0.
\end{equation}
We use the collection of balls $B_j^k$ to perturb the function $\chi_E$.
Recall that for each ball $B_j^k$
there is a function $u_j^k\in BV_c(B_j^k)$ as in~\eqref{gap}. Set
\[
h_k:=\chi_E+\sum_{j=1}^{\infty}u_j^k.
\]
By the $1$-Poincar\'e inequality
\eqref{eq:poincare inequality with zero bry value}
for compactly supported functions,  for all $j\in\N$
\[
  \int_{B_j^k}|u_j^k|\, d\mu\le C r_j^k\, V(u_j^k,B_j^k)
  \le C r_j^k \left[V(\chi_E+u_j^k, B_j^k)+V(\chi_E,B_j^k)\right]
  \le C \frac{2-n^{-1}}{k} V(\chi_E,B_j^k).
\]
Therefore by the pairwise disjointness of the balls in the collection $\mathcal{G}_k$,
\[
\int_X|\chi_E-h_k|\, d\mu\le \sum_{j=1}^{\infty}\int_{B_j^k}|u_j^k|\, d\mu
\le \frac{C}{k} V(\chi_E,X).
\]
Therefore $h_k\to \chi_E$ in $L^1(X)$ as $k\to\infty$.
By the lower semicontinuity of the total variation,
\begin{equation} 
V(\chi_E,X) \leq \liminf_{k \to \infty} V(h_k, X). \label{lsc-2} 
\end{equation}
For ease of notation,
for each $j\in\N$ let 
\[
 G_{k,j}:=\bigcup_{i=1}^j\overline{B}_{i}^k \quad\textrm{and}\quad h_{k,j}:=u+\sum_{i=1}^j\pip_i^k.
\]
Now
\begin{align*}
V(h_{k,j}, G_{k,j})&\le V(h_{k,j},\bigcup_{i=1}^jB_i^k)+
   V(h_{k,j},\bigcup_{i=1}^j\partial B_i^k)\\
   &= V(h_{k,j},\bigcup_{i=1}^jB_i^k)+\sum_{i=1}^jV(\chi_E, \partial B_i^k)\\
   &=V(h_{k,j},\bigcup_{i=1}^jB_i^k)
\end{align*}
and so $V(h_{k,j}, G_{k,j})=V(h_{k,j},\bigcup_{i=1}^jB_i^k)$.
Since $G_{k,j}$ is a closed set, it follows that 
\begin{align*}
 V(h_{k,j},X)&=V(h_{k,j},G_{k,j})+V(h_{k,j},X\setminus G_{k,j})\\
   &=\sum_{i=1}^jV(h_{k,j},B_i^k)+V(\chi_E,X\setminus G_{k,j})\\
   &< (1-n^{-1})\sum_{i=1}^jV(\chi_E,B_i^k)+V(\chi_E,X\setminus G_{k,j})\quad\textrm{by }\eqref{gap}\\
   &=V(\chi_E,X)-n^{-1} V\left(\chi_E,\bigcup_{i=1}^jB_i^k\right).
\end{align*}
Therefore
\begin{equation}\label{eq:estimate for Vhk}
\begin{split}
V(h_k,X) \le \liminf_{j\to\infty}V(h_{k,j},X)
  &\le V(\chi_E,X)-n^{-1} \lim_{j\to\infty} V\left(\chi_E,\bigcup_{i=1}^jB_i^k\right)\\
     &= V(\chi_E,X)-n^{-1} V\left(\chi_E,\bigcup_{B\in\mathcal{G}_k}B\right).
\end{split}
\end{equation}
Set $K_k:=\bigcup_{B\in\mathcal{G}_k}B$ and $F_k:=A_n\setminus K_k$ for each $k\in\N$.
We have $P(E,F_k)=0$ for each $k\in\N$ by \eqref{eq:Vital}.
For $F:=\bigcup_{k=1}^{\infty}F_k$ and $K:=\bigcap_{k=1}^{\infty}K_k$
we then have $P(E,F)=0$.
In light of~\eqref{lsc-2} and \eqref{eq:estimate for Vhk}, we have
\[ 
P(E,X) \leq P(E,X) -\liminf_{k\to\infty}n^{-1} P\left(E,K_k\right)
\leq P(E,X) -n^{-1} P(E,K),
\]
and so $P(E, K) = 0$. Since $A_n\subset F\cup K$,
it follows that $P(E,A_{n})=0$.
This completes the proof.
\end{proof}

A similar analysis can be carried out for functions $u\in BV(X)$ with slightly more involved computations to
obtain analogous asymptotic minimality results for $u$; we do not do so here as we have a stronger
result for $u$ outside its jump and Cantor sets in Theorem~\ref{thm:tangentharm}.

\subsection{Quasiminimality at almost every point}


In this subsection we finally prove the quasiminimality property of the asymptotic limit set $(E)_\8$.

\begin{definition}\label{def:Q-min}
A set $E\subset X$ is said to be {\it $K$-quasiminimal, $K\ge 1$}, if for every $B(x,R) \subset X$ and every $\phi \in BV_c(B(x,R))$ we have
\[
 \frac{1}{K}P(E, B(x,R)) \leq V(\chi_E + \phi, B(x,R)).
\]
Without loss of generality and applying a truncation, one can restrict attention to $\phi$ with values in $[-1,1]$, and such that $\chi_E + \phi$ has values in $[0,1]$.
\end{definition}

The asymptotic minimality of $E$ (Theorem \ref{Asymptotic-FinitePerimeter}) can be upgraded to quasiminimality at generic tangents of the limit set $(E)_\8$. In terms of notation, 
here we only consider the sequence
\[ 
(X_n, d_n, x, \mu_n) := \left(X, \frac{1}{r_n} \cdot d, x, \frac{1}{\mu(B(x, r_n))}\cdot \mu \right) 
\]
under the pointed measured Gromov-Hausdorff convergence, with $r_n \searrow 0$, see the discussion in Section~3.

\begin{theorem}\label{thm:minimaltangent} 
Let $E \subset X$ be a set of finite perimeter. Then, for $P(E, \cdot)$-almost every $x \in X$ and for any space 
$(X_\8, d_\8, x_\8, \mu_\8)$ arising as a pointed measured Gromov-Hausdorff tangent
at $x$, any set $(E)_\8$ 
that arises as an asymptotic limit of $E$ along some sequence $r_n \searrow 0$ is a 
$K$-quasiminimizer. Here, $K$ depends quantitatively on the data.
\end{theorem}

The proof involves lifting Lipschitz functions with small energy to the sequence, and a pasting argument. The desired quasiminimality estimate then follows using Theorem~\ref{Asymptotic-FinitePerimeter} for the lifted sequence. 
We need the following general BV approximation theorem, which is an analog of \cite[Lemma 5.2]{Che} for $p=1$.
We follow the arguments of Cheeger.

\begin{proposition}\label{thm:continuousuppergrad} 
Let $f \in BV(X)$. Then, there exist Lipschitz continuous $f_i$ with bounded Lipschitz continuous upper gradients $v_i$ such that $f_i \to f$ in  $L^1_{\loc}(X)$ and
$v_i \, d\mu \overset{*}{\rightharpoonup} dV(f,\cdot)$.

\end{proposition}

To prove this proposition we define the following auxiliary function. 
For a nonnegative Borel function $g$ on $X$ we set $\mathcal{F}_g\colon X \times X \to [0,\infty]$ to be
\[
\mathcal{F}^X_g(x_1,x_2):=\mathcal{F}_g(x_1, x_2) := \inf_{\gamma} \int_{\gamma} g ~ds,
\]
whenever $x_1,x_2 \in X$. If $x_1=x_2$, we set $\mathcal{F}_g(x_1, x_2) =0$. The infimum is taken over all rectifiable curves $\gamma$ connecting $x_1$ to $x_2$. 
 Note that by the definition of upper gradient~\eqref{eq:upper-grad}, we have that if $g$ is an upper gradient of a 
 function $f\colon X \to \R$, then for every $x,y \in X$,
\[
|f(x)-f(y)| \leq \mathcal{F}_{g}(x,y).
\]
For the proof of the following lemma see~\cite[Lemma 5.18]{Che} or~\cite[pp.~13--14]{HeK}.

\begin{lemma}\label{lem:guppergrad} 
Fix $\eta>0$. Let $g\colon X \to [\eta, \infty)$ be a countably valued lower 
semicontinuous function. Then for $g_n\ge \eta$ an increasing sequence of Lipschitz continuous functions on $X$ converging 
pointwise  $g_n \nearrow g$, we have that for every $x,y \in X$,
\[
\mathcal{F}_{g}(x,y) = \lim_{n \to \infty} \mathcal{F}_{g_n}(x,y).
\]
Moreover, such a sequence $g_n$ exists.
\end{lemma}

\begin{lemma}\label{lem:conv-of-measures}
Let $f\in BV(X)$. Then there is a sequence of Lipschitz functions $f_k$ on $X$ such that $f_k\to f$ in $L^1_{\loc}(X)$ and
$g_k\, d\mu\overset{*}{\rightharpoonup} dV(f,\cdot)$. Here $g_k=\lip f_k$ is an upper gradient of $f_k$.
\end{lemma}

\begin{proof}
By the definition of the total variation we can find a sequence of locally 
Lipschitz functions $f_k$ and upper gradients $g_k=\lip f_k$ such that $f_k\to f$
in $L^1_{\loc}(X)$ and $\lim_k\int_X g_k\, d\mu=V(f,X)$.
Multiplying with suitable cutoff functions if necessary, we can
assume that the $f_k$ are Lipschitz.
For any open set $U\subset X$, we have by the definition of the total variation that
\begin{equation}\label{eq:weak convergence open set}
V(f,U)\le \liminf_{k\to\infty}\int_U g_k\,d\mu.
\end{equation}
On the other hand, for any closed set $F\subset X$ we have
\[
V(f,X) = \lim_{k\to\infty}\int_{X}g_k\,d\mu
  \ge \limsup_{k\to\infty}\int_{F}g_k\,d\mu
  +\liminf_{k\to\infty}\int_{X\setminus F}g_k\,d\mu
  \ge \limsup_{k\to\infty}\int_{F}g_k\,d\mu+V(f,X\setminus F),
\]
where the last inequality again follows by the definition of the total variation.
Thus
\[
\limsup_{k\to\infty}\int_{F}g_k\,d\mu\le V(f,F).
\]
According to a standard characterization of the weak* convergence of Radon measures,
see e.g. \cite[p. 54]{EG}, the above inequality and \eqref{eq:weak convergence open set} together 
give $g_k\, d\mu\overset{*}{\rightharpoonup} dV(f,\cdot)$.
%
%
\end{proof}

\begin{lemma}\label{lem:lsc-to-Lipschitz}
Let $f$ 
be a nonnegative Lipschitz function on $X$ and $g\in L^1_{\loc}(X)$ a bounded 
countably valued
lower semicontinuous upper gradient of $f$. 
Suppose that there is a  $\tau>0$ such that
$g\ge \tau$ on $X$. Then there is a sequence $f_k$ of Lipschitz continuous functions on $X$ with 
$f_k\to f$ in $L^1_{\loc}(X)$ and bounded 
Lipschitz continuous upper gradients $g_k$ of $f_k$ such that $g_k\to g$ in $L^1_{\loc}(X)$ and $g_k$ 
monotone increases to $g$ everywhere on $X$, 
and $g_k\ge \tau$ for each $k$.
\end{lemma}


\begin{proof}
Since $g$ is lower semicontinuous,
we can find a sequence of Lipschitz continuous functions $g_k\ge \tau$ on $X$ such that 
$g_k\to g$ in $L^1_{\loc}(X)$ and in addition $g_k\le g_{k+1}\le g$ on $X$ for each $k\in\N$.
By Lemma~\ref{lem:guppergrad}  we know that $\mathcal{F}_g=\lim_k\mathcal{F}_{g_k}$ pointwise everywhere on $X\times X$.

Next, we fix $x_0\in X$ and
for each positive integer $i$ let $\widehat{A_i}$ be a maximal $1/i$-net of $X$ such that 
$\widehat{A_i}\subset \widehat{A_{i+1}}$ for each $i\in\N$, and
let $A_i=\widehat{A_i}\cap B(x_0,2i)$. Then $A_i\subset A_{i+1}$, and
by the doubling property of $\mu$ we know that $A_i$ is a finite set for each $i$.
As $g$ is bounded, we can also ensure that 
each $g_k\le M$ and $g\le M$ on $X$ for some positive $M$. Therefore for each $y\in X$ we 
know that $\mathcal{F}_g$ and $\mathcal{F}_{g_k}$ are $MC$-Lipschitz where $C$ is the quasiconvexity constant of $X$. Now,
taking inspiration from the McShane extension (see also~\cite{Che}), we set
\[
f_k(x):=\inf\{f(y)+\mathcal{F}_{g_{k}}(x,y)\, :\, y\in A_k\}.
\]
Then $f_k$ is also $MC$-Lipschitz on $X$.  A standard argument
(see e.g. \cite[p. 384]{HKST})
shows that $g_{k}$ is an upper gradient of $f_k$. 

For $x\in \bigcup_nA_n$, we choose $n\in\N$ such that $x\in A_n$; then for $k\ge n+1$ we see that $x\in A_k$. It follows that
$f_k(x)\le f(x)$. If $y\in X\setminus B(x,L)$ for some $L>0$ then as $f$ is nonnegative, 
$f(y)+\mathcal{F}_{g_{k}}(x,y)\ge L\tau$; thus to 
obtain $f_k(x)$ it suffices to look only at $y\in A_{k}\cap B(x,L)$ where $L=[1+f(x)]/\tau$. Let $y_k\in A_k\cap B(x,L)$ such that
\[
k^{-1}+f_k(x)\ge f(y_k)+\mathcal{F}_{g_{k}}(x,y_k).
\]
Then the sequence $(y_k)$ lies in the compact set $\overline{B}(x,L)$ and hence has a subsequence $y_{k_j}$ converging to
some $y_\8\in\overline{B}(x,L)$. Thus $f(x)\ge \lim_{k\to\infty}f_k(x)\ge f(y_\8)+\lim_{k\to\infty}\mathcal{F}_{g_{k}}(x,y_k)$.
Observe that
\[
|\mathcal{F}_{g_{k}}(x,y_k)-\mathcal{F}_{g_{k}}(x,y_\8)|\le MC\, d(y_k,y_\8).
\]
It then follows from Lemma~\ref{lem:guppergrad} that 
\[
f(x)\ge \lim_{k\to\infty}f_k(x)\ge  f(y_\8)+\lim_{k\to\infty}\mathcal{F}_{g_{k}}(x,y_\8)=f(y_\8)+\mathcal{F}_g(x,y_\8)\ge f(x),
\]
and it then follows that $\lim_{k\to\infty} f_k(x)=f(x)$. Now the uniform Lipschitz continuity of $f_k$, $k\in\N$ and $f$ shows that
$\lim_kf_k=f$ pointwise on $X$. An appeal to the Lebesgue dominated convergence theorem 
(and the fact that $f_k\le \Vert f\Vert_{L^\infty(B)}+ Mk<\infty$
on the ball $B=B(x_0,k)$)
yields the convergence also in $L^1_{\loc}(X)$.
\end{proof}

The above lemmas allow us now to prove Proposition~\ref{thm:continuousuppergrad}.

\begin{proof}[Proof of Proposition~\ref{thm:continuousuppergrad}]
By Lemma~\ref{lem:conv-of-measures} we obtain a sequence $f_k$ of
Lipschitz
functions on $X$ with $f_k\to f$ in $L_{\loc}^1(X)$ and 
upper gradients $g_k=\lip f_k$ of $f_k$ such that $g_k\, d\mu\overset{*}{\rightharpoonup} dV(f,\cdot)$. 
Note that each
$g_k$ is bounded.
By the Vitali-Carath\'eodory theorem, see e.g. \cite[p. 108]{HKST},
for each $k$ we can find a bounded countably valued lower semicontinuous 
function $g_k^\prime\ge g_k$ such that $\Vert g_k^\prime-g_k\Vert_{L^1(X)}\to 0$ as $k\to\infty$. Note that automatically $g_k^\prime$ is
also an upper gradient of $f_k$. Moreover, we now have
$g_k^\prime\, d\mu\overset{*}{\rightharpoonup} dV(f,\cdot)$, and so 
we also have $[g_k^\prime+k^{-1}]d\mu\overset{*}{\rightharpoonup} dV(f,\cdot)$. 

Next we apply Lemma~\ref{lem:lsc-to-Lipschitz} to obtain
bounded Lipschitz functions $v_k$ and Lipschitz functions $F_k$ such that
$v_k$ is an upper gradient of $F_k$, $F_k\to f$ in $L^1_{\loc}(X)$, and $v_k-[g_k^\prime+k^{-1}] \to 0$ in $L^1_{\loc}(X)$
as 
$k\to\infty$. It follows then also that
$v_k\, d\mu\overset{*}{\rightharpoonup} dV(f,\cdot)$, completing the proof of the proposition. 
\end{proof}


We will need the following lemma from Keith~\cite[Proposition 4]{Kei03}, see
also~\cite[proof of Proposition~2.17]{HeK}. 
This lemma is a simple consequence of the Arzel\`a-Ascoli theorem together with the lower semicontinuity of $g$.
In the lemmas below we will consider curves
to be arc length parametrized.

\begin{lemma}{{\rm\bf \cite[Proposition~4]{Kei03}}}\label{lem:semicontcurv} 
 If $Z$ is a proper space, $g \colon Z \to \R$ a nonnegative lower semicontinuous function, $L>0$, and 
 $\gamma_n$ a sequence of curves in $Z$ with length at most $L$ and are contained in a fixed compact subset of $Z$, then there 
 exists a rectifiable curve $\gamma_{\8}$ so that a subsequence of $\gamma_n$ converges to $ \gamma_\8$  
 uniformly. For such $\gamma_\8$ we also have that 
\[
\int_{\gamma_\8} g \, ds \leq \liminf_{n \to \infty} \int_{\gamma_n} g\, ds.
\]
\end{lemma}

As a corollary, we obtain the following.

\begin{lemma}\label{lem:semicontfun} 
Let $g \colon Z \to [\tau,\infty)$ be a nonnegative lower semicontinuous function on 
a proper space $Z$ for some $\tau>0$, and assume that $x_n \to x$ and $y_n \to y$ are 
sequences of points in $Z$. Then,
\begin{equation}\label{eq:semicont-Fg}
\mathcal{F}_{g}(x,y) \leq \liminf_{n \to \infty}  \mathcal{F}_{g}(x_n,y_n).
\end{equation}
\end{lemma}

Note that we avoid assuming $Z$ has any rectifiable curves, or that it is quasiconvex. This is 
necessary for our application where $Z$ is the 
proper metric space into which the sequence of scaled spaces $X_i$ and the tangent space $X_\8$ embed isometrically
as described in the latter part of Remark~\ref{rmk:control-eps}.

\begin{proof}
If the limit infimum on the right hand side of~\eqref{eq:semicont-Fg} is infinite, there is nothing to prove. So 
we will assume that it is finite.
By passing to a subsequence, we can 
assume that there is some real number $M>0$ such that 
$\mathcal{F}_{g}(x_n,y_n)\le M$ for all $n$. Then for every $0< \epsilon < M$, there exist curves $\gamma_n$ connecting 
$x_n$ and $y_n$ such that

\[
\tau \ell(\gamma_n)\le \int_{\gamma_n} g ~ds \leq  \mathcal{F}_{g}(x_n,y_n)  + \epsilon \leq 2M.
\]
Since $\gamma_n$ connects $x_n$ to $y_n$, and these 
converge, respectively, to $x$ and $y$, the curves $\gamma_n$ lie, for sufficiently large $n$, in the closed ball 
$\overline{B(x,M+2M/\tau)}$ which is compact.  
Then, by Lemma~\ref{lem:semicontcurv}, by taking a subsequence if necessary, the sequence $\gamma_n$ converges to 
some curve $\gamma_\8$, and 
\[
\mathcal{F}_{g}(x,y) \leq \int_{\gamma_\8} g \, ds \leq \liminf_{n \to \infty}  \int_{\gamma_n} g \, ds 
\leq  \liminf_{n \to \infty}   \mathcal{F}_{g}(x_n,y_n)  + \epsilon.
\]
Since this holds for every small $\epsilon>0$ the claim follows.
\end{proof}

\begin{lemma}\label{lem:lifting} 
Let $(X_i, d_i, x_i, \mu_i) \to (X_\8, d_\8, x_\8, \mu_\8)$ be a sequence of scaled (from $X$) metric measure 
spaces converging in the pointed measured Gromov-Hausdorff sense. If $f$ is a nonnegative 
Lipschitz function on $X_\8$, with a
bounded Lipschitz upper gradient $v$, then there exists a 
subsequence, also denoted $(X_i,d_i,x_i, \mu_i)$, and uniformly Lipschitz continuous functions
$f_i$ with Lipschitz continuous upper gradients $v_i$ on $X_i$ such that
\[
v_{i} \, d\mu_{i} \overset{*}{\rightharpoonup} v \, d\mu_\8,
\]
and $f$ is a limit function of $f_{i}$ in the sense of~\eqref{eq:zoomlimit-2}. 
\end{lemma}


\begin{proof} 
Without loss of generality, we can assume that $v\ge \tau$ for some positive $\tau$, 
since otherwise we can obtain the result by considering $\max\{v,1/k\}$ instead of $v$
for each positive integer $k$, and then complete the proof with the help of a diagonalization
argument, letting $k\to\infty$.

Let $\hat{v}\colon Z\to\R$ be a McShane extension of the Lipschitz function $v\circ\iota\vert_{\iota(X_\8)}^{-1}$
on $\iota(X_\8)$ to  the entirety of $Z$.  Also, such an extension 
can be chosen to be bounded and so that $\hat{v} \geq \tau$.
Let $v_i := \hat{v} \circ \iota_i\colon  X_i\to\R$. 

Next, let $\hat{f} \colon Z \to \R$ be constructed similarly, by first setting $\hat{f}(z)  := f \circ \iota^{-1}(z)$ for 
$z \in \iota(X_\8)$, and then taking a McShane extension to $Z$. We can choose $\hat{f}$ to be nonnegative.  
Next we construct the functions $f_i\colon X_i\to\R$ so that $v_i$ is an upper gradient of $f_i$ as follows.
For $x \in X_i$ we set
\[
f_i(x) := \inf_{y  \in X_i} [\hat{f}(\iota_i(y)) + \mathcal{F}^{\iota_i(X_i)}_{\hat{v}}(\iota_i(y),\iota_i(x))]
  = \inf_{y  \in X_i} [\hat{f}(\iota_i(y)) + \mathcal{F}^{X_i}_{v_i}(y,x)].
\]
For ease of notation, we set $\mathcal{F}_{v_i}(x,y):=\mathcal{F}^{\iota_i(X_i)}_{\hat{v}}(\iota_i(y),\iota_i(x))$
for $x,y\in X_i$.
From the definition of $f_i$ it is clear that $f_i(x) \leq  \hat{f}(\iota_i(x))$ for each $x\in X_i$. 
Also, $f_i$ is nonnegative, and has  $v_i$ as an upper gradient. 

 We will now show that $f$ is a limit function of $f_i$. To do so, we need to show for every $r>0$,
\[
 \lim_{i \to \infty} \| f - f_{i}\circ \phi_{i} \|_{L^\8 (B_{X_\8}(x_\8, r))}=0,
 \]
where $\phi_i$ are the approximating maps from Definition~\ref{pGH}.  Suppose this is not the case. 
Then there is some $r>0$ and some $\delta>0$ such that, by passing to a subsequence if needed, 
we have 
\begin{equation}\label{eq:contradicting-hyp}
\liminf_{i \to \infty} \| f - f_{i}\circ \phi_{i} \|_{L^\8 (B_{X_\8}(x_\8, r))}> \delta.
\end{equation}
Thus, for each $i$ there is a point $x_i \in B_{X_\8}(x_\8, r)$ such that 
\begin{equation}\label{eq:contradicting-hyp cons}
| f(x_i)- f_{i}(\phi_{i}(x_i))| > \delta.
\end{equation}
Since $X_\8$ is proper and $x_i\in B_{X_\8}(x_\8,r)$ for all $i$, there is a subsequence,
also denoted with the index $i$, such that $x_i\to x\in X_\8$.
Fix $\delta>0$. Then from the definition of 
$f_{i}(\phi_{i}(x_i))$, we have $y_i \in X_i$ such that 
\begin{equation}\label{eq:using def of fi}
| f_{i}(\phi_{i}(x_i)) - \hat{f}(\iota_i(y_i)) - \mathcal{F}_{v_i}(y_i,\phi_{i}(x_i)) | \leq \delta/4.
\end{equation}
Combining the above with \eqref{eq:contradicting-hyp cons} we get
\[
| f(x_i)-\hat{f}(\iota_i(y_i)) - \mathcal{F}_{v_i}(y_i,\phi_{i}(x_i))| \geq \delta/2.
\]
By \eqref{eq:using def of fi} we have
$\hat{f}(\iota_i(y_i))+ \mathcal{F}_{v_i}(y_i,\phi_{i}(x_i))
\le f_{i}(\phi_{i}(x_i))+\delta/4
\le \hat{f}(\iota_i(\phi_i(x_i)))+\delta/4$, 
and so the triangle inequality gives
\begin{equation}\label{eq:badlimest}
 | f(x_i)-\hat{f}(\iota_i(\phi_i(x_i)))| +\hat{f}(\iota_i(\phi_i(x_i)))-\hat{f}(\iota_i(y_i)) - \mathcal{F}_{v_i}(y_i,\phi_{i}(x_i))\ge \delta/4.
\end{equation}
For the first term, note that $f(x_i) = \hat{f}(\iota(x_i))$, and from~\eqref{eq:compatibile-phi-iota} we get
 $\lim_{i \to \infty} d_Z(\iota(x_i), \iota_i(\phi_i(x_i))) = 0$, and thus from the Lipschitz continuity of $\hat{f}$,
\begin{equation} \label{eq:limfhat}
\lim_{i \to \infty} | \hat{f}(\iota(x_i))-\hat{f}(\iota_i(\phi_i(x_i)))| =0.
\end{equation}
Since $\lim_ix_i=x$, we also have  
\begin{align}\label{eq:bdd-phi-xi}
d_Z(\iota_i(\phi_i(x_i)),\iota(x))&\le d_Z(\iota_i(\phi_i(x_i)),\iota(x_i))+d_Z(\iota(x_i),\iota(x))\notag\\
 &=d_Z(\iota_i(\phi_i(x_i)),\iota(x_i))+d_{X_\8}(x_i,x)\to 0\quad \text{ as }i\to\8.
\end{align}
Therefore the sequence  of real numbers
$\hat{f}(\iota_i(\phi_{i}(x_i)))$ is bounded, that is, there is some
$M>\delta>0$ such  that 
$\sup_i \hat{f}(\iota_i(\phi_{i}(x_i)))\le M$.
The functions $v_i$ are bounded from below by $\tau$ and $f$ is nonnegative. Therefore, if $d(\phi_{i}(x_i),y_i) >   2M/\tau$, then
\[
\hat{f}(\iota_i(y_i)) + \mathcal{F}_{v_i}(y_i,\phi_{i}(x_i)) \geq    \mathcal{F}^{X_i}_{v_i}(y_i,\phi_{i}(x_i)) 
\geq 2M > \hat{f}(\iota_i(\phi_{i}(x_i)))+\delta \ge f_{i}(\phi_{i}(x_i))+\delta,
\]
which would violate the choice of $y_i$, \eqref{eq:using def of fi}. Hence
we must have $d(\phi_{i}(x_i),y_i) \leq   2M/\tau$. 
As the sequence $\iota_i(\phi_i(x_i))$ lies in a ball, in $Z$, centered at $\iota(x)$ by~\eqref{eq:bdd-phi-xi}, we see then
that the sequence $\iota_i(y_i)$ also lies in a ball centered at $\iota(x)$.
Therefore, by the properness of $Z$, there is a subsequence, also denoted with the index $i$, and a point
$\hat{y}\in Z$ such that $\lim_i\iota_i(y_i)=\hat{y}$. As $y_i\in X_i$ and $X_i$ converges to the metric space $X_\8$, it follows
that $\hat{y}=\iota(y)$ for some $y\in X_\8$.
Then by Lemma~\ref{lem:semicontfun} we get
$\mathcal{F}^Z_{\hat{v}}(\iota(x),\iota(y)) \leq  \liminf_{i\to \infty} \mathcal{F}^Z_{\hat{v}}(\iota_i(\phi_i(x_i)),\iota_i(y_i))$.
Note that 
$\mathcal{F}_{v_i}(y_i,\phi_i(x_i))=\mathcal{F}^{\iota_i(X_i)}_{\hat{v}}(\iota_i(y_i),\iota_i(\phi_i(x_i)))$, which is not the
same as $\mathcal{F}^Z_{\hat{v}}(\iota_i(y_i),\iota_i(\phi_i(x_i)))$. However, we have that
$\mathcal{F}^Z_{\hat{v}}(\iota_i(y_i),\iota_i(\phi_i(x_i)))\le \mathcal{F}^{\iota_i(X_i)}_{\hat{v}}(\iota_i(y_i),\iota_i(\phi_i(x_i)))$.
Now by~\eqref{eq:badlimest} and \eqref{eq:limfhat}, we obtain
\begin{align*}
\tfrac{\delta}{4}+\mathcal{F}^Z_{\hat{v}}(\iota(x),\iota(y)) &\le 
\tfrac{\delta}{4}+\liminf_i\mathcal{F}^Z_{\hat{v}}(\iota_i(y_i),\iota_i(\phi_i(x_i)))\\
&\le \lim_i[\hat{f}(\iota_i(\phi_i(x_i)))-\hat{f}(\iota_i(y_i))]=\hat{f}(\iota(x))-\hat{f}(\iota(y))=f(x)-f(y).
\end{align*}
We now use the specific structure of $Z$; by~\cite{Her}, we can choose $Z$ to be 
the completion of pairwise disjoint union of $X_i$, $i\in\N$. With such a choice, 
it follows that if $\gamma$ is a non-constant rectifiable curve in $Z$, then either $\gamma$ lies entirely in $\iota_i(X_i)$ for some
positive integer $i$, or else $\gamma$ lies entirely in $\iota(X_\8)$. It follows that 
\[
\mathcal{F}^Z_{\hat{v}}(\iota(x),\iota(y))=\mathcal{F}^{\iota(X_\8)}_{\hat{v}}(\iota(x),\iota(y))
    =\mathcal{F}^{X_\8}_v(x,y).
\]
Hence from the above inequality we obtain 
\[
\mathcal{F}^{X_\8}_v(x,y)<\tfrac{\delta}{{4}}+\mathcal{F}^{X_\8}_v(x,y)\le f(x)-f(y)\le |f(x)-f(y)|,
\]
which is not possible as $v$ is an upper gradient of $f$. Thus~\eqref{eq:contradicting-hyp} is false, and so
$f=\lim_if_i$ as desired.

Finally, we show that
$v_{i} \, d\mu_{i} \overset{*}{\rightharpoonup} v \, d\mu_\8$ as follows.
Pick a test function $\phi\in C_c(Z)$. Then also $\phi\hat{v}\in C_c(Z)$.
Using this fact and the fact that $\iota_{i,*}\mu_{i}\overset{*}{\rightharpoonup}
\iota_*\mu_{\infty}$,
we get
\[
\lim_{i\to\infty}\int_Z \phi \,\iota_{i,*}(v_i\,d\mu_{i})=
\lim_{i\to\infty}\int_Z \phi \,\hat{v}\,d\iota_{i,*}\mu_{i}
=\int_Z \phi \,\hat{v}\,d \iota_*\mu_{\infty}
=\int_Z \phi \,\iota_{*}(v\,d\mu_{\infty}).
\]
\end{proof}



We also need the following lemma, which stitches two given BV functions along an annulus to yield a BV function
whose BV energy is controllable.

\begin{lemma}{{\rm\bf{\cite[Lemma~3.3]{Mir}}}}\label{lem:merging} 
Let $f \in BV(X)$, $x\in X$, $0<a<b\le R$, and $g \in BV(B(x, b))$. Then 
there exists a $2/(b-a)$-Lipschitz function $\eta : X \to [0,1]$ with compact support in $B(x, b)$, and such that $\eta=1$ 
on $B(x,a)$ such that $h= \eta g + (1-\eta) f\in BV(X)$ with
\[
V(h, B(x,R)) \leq V(f,B(x,R) \setminus \overline{B(x,a)}) + V(g, B(x,b)) 
  + \frac{2}{b-a} \int_{B(x, b) \setminus B(x, a)} |f-g| ~d\mu.
\]
\end{lemma}

Finally, we can conclude the proof of Theorem~\ref{thm:minimaltangent}. 


\begin{proof}[Proof of Theorem~\ref{thm:minimaltangent}]
Let $x \in\Sigma_\gamma$ be a point where the conclusions of Theorem~\ref{Asymptotic-FinitePerimeter}, 
Lemma~\ref{lem:perim-growth-bounds} and 
Theorem~\ref{thm:main3} hold. We will show that the corresponding asymptotic set $(E)_\8$ is 
$K$-quasiminimal for some $K$, which will be determined at the end of the proof. 
Since $P(E, \cdot)$-almost every $x \in X$ is such a point, this concludes the proof. 
Let $R>0$,  $z\in X_\8$,
and $\pip\in BV_c(B_{X_\8}(z,R))$. By slightly decreasing $R$ if necessary, we can assume that 
$P((E)_\8, \partial B_{X_\8}(z,R))=0$.

From Theorem~\ref{Asymptotic-FinitePerimeter}, there is some $r_0>0$ such that
that for every $r_0>r>0$ there is some positive $\eps_r$ such that 
$\lim_{r\to 0^+}\eps_r=0$ and whenever 
$\psi\in BV_c(B_{X}(x,r))$, we have
\begin{equation}\label{eq:asymp-to-local}
P(E,B_X(x,r))\le (1+\eps_r)\, V(\chi_E+\psi, B_X(x,r)).
\end{equation}
By a standard truncation argument, we can assume without loss of generality that 
$0\leq  \chi_{(E)_\8} + \pip \leq 1$.
By Proposition~\ref{thm:continuousuppergrad} we can find a sequence of Lipschitz function--Lipschitz upper gradient
pairs $f_i, v_i$ on $X_\8$, with each $v_i$ bounded,
such that $f_i\to \chi_{(E)_\8}+\pip$ in $L^1_\loc(X_\8)$ and 
$v_i\, d\mu_\8 \overset{*}{\rightharpoonup} dV(\chi_{(E)_\8}+\pip,\cdot)$. Next, for each positive integer $i$
we apply Lemma~\ref{lem:lifting} to obtain lifts $f_{i,n}, v_{i,n}$ to $X_n$ such that 
$v_{i,n}\, d\mu_n\overset{*}{\rightharpoonup} v_i\, d\mu_\8$ and $f_{i,n}\to f_i$. 
Further, by truncating each $f_i$ and $f_{i,n}$, we can also assume $0\leq f_i, f_{i,n} \leq 1$.

By passing to a subsequence of $(X_n, d_n, x,\mu_n)$ if necessary, 
with $\rho$ fixed and chosen in the interval $[2[R+d_{X_\8}(z,x_\8)],3[R+d_{X_\8}(z,x_\8)]]$,
we have $\rho\, r_n<r_0$ and that 
\[
\sup_{y,w\in B_{X_\8}(x_\8,\rho)}|d_{n}(\phi_n(y),\phi_n(w))-d_{X_\8}(y,w)|<\frac1n,
\]
and
\[
B_{n}(x,\rho)\subset \bigcup_{y\in \phi_n(B_{X_\8}(x_\8,\rho+1/n))}B_{n}(y,1/n).
\]
For each $n$ we set $x_n:=\phi_n(z)$. By choosing $\rho$ appropriately, we can also ensure
\begin{equation}\label{eq:not-charge-bdy-pi}
 \pi_{\infty}(\partial B_{X_\8}(x_\8,\rho))=0.
\end{equation}

Then by the above, 
\[
\frac{d_X(x,x_n)}{r_n}=d_{n}(x,x_n)\le d_{X_\8}(x_\8,z)+\frac1n.
\]
Fix $\tau \in (0,1)$. We now use Lemma~\ref{lem:merging}
to stitch $f_{i,n}$ on $B_{n}(x_n,R)$ to $\chi_E$ on $B_{n}(x,\rho)\setminus B_{n}(x_n,[1+\tau]R)$
using the Lipschitz function $\eta_n$ to obtain $h_{i,n}:=\eta_n f_{i,n}+(1-\eta_n)\chi_E$. Then, since 
$\rho\, r_n<r_0$, we know that
\[
P(E,B_X(x,\rho\, r_n))\le (1+\eps_{ \rho \, r_n})\, V(h_{i,n},B_X(x,\rho\, r_n)).
\]
Note by Lemma~\ref{lem:merging} that
\begin{align*}
V(h_{i,n},B_{n}(x,\rho))&\le P_n(E, B_{n}(x,\rho)\setminus B_{n}(x_n,R))
  +V(f_{i,n}, B_{n}(x_n,[1+\tau]R))\\
    &\qquad+\frac{2}{\tau R}\, \int_{B_{n}(x_n,[1+\tau]R)\setminus B_{n}(x_n,R)}|f_{i,n}-\chi_E|\, d\mu_n.
\end{align*}
Note that $h_{i,n} - \chi_E$ has compact support on $B_n(x,\rho)$ for large enough $n$ since we 
can ensure $B_{n}(x_n,[1+\tau]R) \subset B_n(x,\rho)$. Combining this with the (asymptotic) minimality of 
$\chi_E$ at $x$ as explained above, we obtain 
that
\begin{align*}
P_n(E,B_{n}(x,\rho))& \le [1+\eps_{\rho r_n}] \left( P_n(E, B_{n}(x,\rho)\setminus B_{n}(x_n,R))+\int_{B_{n}(x_n,[1+\tau]R)}v_{i,n}\, d\mu_n \right. \\ 
    &\left. \qquad+\frac{2}{\tau R}\, \int_{B_{n}(x_n,[1+\tau]R)\setminus B_{n}(x_n,R)}|f_{i,n}-\chi_E|\, d\mu_n \right).
\end{align*}
In the above, we have also used the fact that as $v_{i,n}$ is an upper gradient of $f_{i,n}$, we have
$dV(f_{i,n},\cdot)\le v_{i,n}\, d\mu_n$. 

Recall that $\chi_E$ is either $0$ or $1$ on $X_\8$, and 
$0\leq f_{i,n}, f_i \leq 1$, 
and so we have
\[
|f_{i,n}-\chi_E| = (1-f_{i,n})\chi_E + (1-\chi_E)f_{i,n}
\]
and
\[
|f_{i}-\chi_{(E)_\8}| = (1-f_{i})\chi_{(E)_\8} + (1-\chi_{(E)_\8})f_{i}.
\]
Thus, since $\chi_E\, d\mu_n \overset{*}{\rightharpoonup} \chi_{E_\8}\, d\mu_\8$
and since $\mu_\8$ gives measure zero to every sphere due to the geodesic property
and doubling of $X_\8$ (recall \eqref{eq:annular decay property}), we get
\begin{align*}
& \lim_{n\to \infty} \int_{B_{n}(x_n,[1+\tau]R)\setminus B_{n}(x_n,R)}|f_{i,n}-\chi_E|\, d\mu_n  \\
&\qquad= \lim_{n\to \infty} \int_{B_{n}(x_n,[1+\tau]R)\setminus B_{n}(x_n,R)}[(1-f_{i,n})\chi_E + (1-\chi_E)f_{i,n}] \, d\mu_n  \\
 &\qquad=\int_{B_{X_\8}(z,[1+\tau]R)\setminus B_{X_\8}(z,R)}[(1-f_{i})\chi_{(E)_\8} + (1-\chi_{(E)_\8})f_{i}] \, d\mu_\8  \\
 &\qquad=\int_{B_{X_\8}(z,[1+\tau]R)\setminus B_{X_\8}(z,R)}|f_i-\chi_{(E)_\8}|\, d\mu_\8.
\end{align*}
Now letting $n\to\infty$ and using~\eqref{eq:not-charge-bdy-pi}, we obtain 
\begin{align*}
\pi_\8(B_{X_\8}(x_\8,\rho))&\le \pi_\8(B_{X_\8}(x_\8,\rho)\setminus B_{X_\8}(z,R)) 
+\int_{B_{X_\8}(z,[1+\tau]R)}v_i\, d\mu_\8\\
&\qquad+ \frac{2}{\tau R}\int_{B_{X_\8}(z,[1+\tau]R)\setminus B_{X_\8}(z,R)}|f_i-\chi_{(E)_\8}|\, d\mu_\8.
\end{align*}
Thus we get
\[
\pi_\8(B_{X_\8}(z,R))\le \int_{B_{X_\8}(z,[1+\tau]R)}v_i\, d\mu_\8
 +\frac{2}{\tau R}\int_{B_{X_\8}(z,[1+\tau]R)\setminus B_{X_\8}(z,R)}|f_i-\chi_{(E)_\8}|\, d\mu_\8.
\]
Now letting $i\to\infty$ gives
\[
\pi_\8(B_{X_\8}(z,R))\le V(\chi_{(E)_\8}+\pip, B_{X_\8}(x,[1+2\tau]R)),
\]
where we used the fact that $f_i\to\chi_{(E)_\8}$ in $L^1_{\loc}(X_\8)$
and $v_i\, d\mu_\8 \overset{*}{\rightharpoonup} dV(\chi_{(E)_\8}+\pip,\cdot)$.
Now letting $\tau\to 0$ and finally using the assumption
$P((E)_\8, \partial B_{X_\8}(z,R))=0$ we obtain
\[
\pi_\8(B_{X_\8}(z,R))\le V(\chi_{(E)_\8}+\pip, B_{X_\8}(z,R)).
\]
Now by Theorem~\ref{thm:main3} we have
\[
P((E)_\8, B_{X_\8}(z,R))\le C\, V(\chi_{(E)_\8}+\pip, B_{X_\8}(z,R)),
\]
where $C$ is the comparison constant that connects $\pi_\8$ to $P((E)_\8,\cdot)$. 
Thus choosing $K=C$ yields the desired outcome. This completes the proof.
\end{proof}

\subsection{Concluding remarks} In Section~4 we have shown that any asymptotic limit, at $\mu$-almost every point, 
of a BV function is a function of least gradient on a corresponding tangent space $X_\8$
and is Lipschitz continuous with a constant minimal $p$-weak upper gradient. In Section~6 we have shown that given
a set $E$ of finite perimeter in $X$, at $\mathcal{H}$-almost every point of its measure-theoretic boundary we have the 
existence of an asymptotic limit set $(E)_\8\subset X_\8$ such that this asymptotic limit set is of quasiminimal boundary
surface (that is, $\chi_{(E)_\8}$ is of quasi-least gradient).

\begin{remark}
If $u\in BV(X)$, from the co-area formula we know that for almost every $t\in\mathbb{R}$ its super-level set 
\[
E_t:=\{x\in X\, :\, u(x)>t\}
\] 
is of finite perimeter in $X$. Let $\R_F$ be the collection of all $t\in\R$ for which $E_t$ is
of finite perimeter, and let $A\subset\R_F$ be a countable dense subset of $\R_F$, and for each $t\in A$ let
$K_t$ be the collection of all points in $X$ at which the conclusion of Theorem~\ref{thm:minimaltangent} fails; then
$\mathcal{H}(\bigcup_{t\in A}K_t)=0$. Let $x\in S_u\setminus\bigcup_{t\in A}K_t$, where $S_u$ is the jump set of $u$.
Note that if $x\in X\setminus\partial^*E_t$, then for every tangent space $X_\8$ based at that point,
the corresponding set $(E_t)_\8$ is either all of $X_\8$ or is empty, and hence does satisfy the conclusion of
Theorem~\ref{thm:minimaltangent}. Thus we have here that $K_t\subset\partial^*E_t$.
A Cantor diagonalization argument gives us for each $t\in A$ an asymptotic limit $(E_t)_\8\subset X_\8$, with 
$X_\8$ a tangent space to $X$ based at $x$, of the set $E_t$. We know then that each $(E_t)_\8$ is of quasiminimal
boundary in $X_\8$ in the sense of~\cite{KKLS}, with the quasiminimality constant $K$ independent of
$t$. Moreover, note that if $t_1,t_2\in A$ such that $t_1<t_2$, then $E_{t_2}\subset E_{t_1}$ and so by
the construction of $(E_t)_\8$ we have that $(E_{t_2})_\8\subset (E_{t_1})_\8$. 
Indeed, by the definition of $(E)_\8$ from 
the discussion before Lemma~\ref{lem:consequences of annular decay}, we have that when $z\in X_\8$ for which
$z\in (E_{t_2})_\8$, we have $\mu_\8^{E_{t_2}}\le \mu_\8^{E_{t_1}}$ on $X_\8$ 
(because $E_{t_2}\subset E_{t_1}\subset X$), and so
\[
1=\lim_{r\to 0^+}\frac{\mu_\8^{E_{t_2}}(B_{X_\8}(z,r))}{\mu_\8(B_{X_\8}(z,r))}
 \le \lim_{r\to 0^+}\frac{\mu_\8^{E_{t_1}}(B_{X_\8}(z,r))}{\mu_\8(B_{X_\8}(z,r))}\le 1,
\]
and so we must have $z\in (E_{t_1})_\8$. In this discussion, recall that we have fixed $x\in S_u\setminus\bigcup_{t\in A}K_t$.
We can now set
\[
u_\8(z):=\sup\{t\in A :\, z\in (E_t)_\8\}.
\]
An argument as in the proof of~\cite[Theorem~4.11]{LMSS} tells us that $u_\8$ is of quasi-least gradient in
$X_\8$. It would be interesting to know in which sense, if any, is this $u_\8$ an asymptotic limit of $u$ at $x$.
\end{remark}

The limit set $E_\8$ is a quasi-minimizer according to Theorem~\ref{thm:minimaltangent} , in contrast to
the minimizer property of the limit $u_\8$ of $u$ at the aboslutely continuous point of $\Vert Du\Vert$, 
see Theorem~\ref{thm:tangentharm}. The next example shows that this disconnect is real and is not
an artifact of our proof.

\begin{Example}\label{ex:nonmin} 
For positive integers $n$ let $a_n=1/n!$ and $b_n=-a_n$.
Let $X=\R$ be equipped with the Euclidean metric and with
a weighted measure $d\mu=w d\mathcal{L}^1$. Let us choose the weight $w$ so that 
\[
 w(x)=\begin{cases} 2&\text{ if }b_{2n-1}<x\le b_{2n}\text{ or }a_{2n}\le x<a_{2n-1},\\
              1&\text{ otherwise. }\end{cases}
\]
Then if we choose the base point $x=0$ and the sequence of scales
$r_n=1/(2n-1)!$, we can see that the limit space $X_\infty=\R$ is equipped with
the measure $\mu_\infty$ given by 
$d\mu_\infty=(\chi_{[-1,1]}+\tfrac12\chi_{\R\setminus[-1,1]})d\mathcal{L}^1$.
If we take $E=(-\infty,0]$, then $E$ is of finite perimeter with perimeter measure $P(E,\cdot)$
the Dirac measure supported at $0$. The limit set $E_\infty=(-\infty,0)$ is quasiminimal, but
is not a minimal set as $F:=(-\infty,1)$ has a smaller perimeter measure; here,
the perimeter measure $P_\infty(E_\infty,X_\infty)=1$ whereas $P_\infty(F,X_\infty)=\tfrac12$, and note
that $E_\infty\Delta F$ is a relatively compact subset of $X_\infty$. 
\end{Example}

\begin{remark}
In Definition~\ref{def:Q-min} of quasiminimality we used balls $B(x,R)$. The study undertaken in~\cite{KKLS} is applicable
to functions satisfying this definition; however, the notion of quasiminimality given in~\cite{KKLS} is slightly stronger,
namely whenever $\varphi$ is a compactly supported $BV$ function on $X$, we have
\[
 V(u,\text{supt}(\pip))\le K\, V(u+\pip, \text{supt}(\pip)).
\]
The proof given in Subsection~6.2 can be easily adapted to prove that $\chi_{(E)_\8}$ satisfies this stronger version,
but the proof gets messy, and hence we gave the relatively more transparent proof showing that $\chi_{(E)_\8}$ satisfies
Definition~\ref{def:Q-min}. To prove the stronger quasiminimality criterion of~\cite{KKLS}, one first modifies the
stitching lemma (Lemma~\ref{lem:merging}) by replacing $B(x,a), B(x,b)$ with open sets $U,V$ with $U\Subset V$ and considering
$\eta$ to be a Lipschitz function with $\eta=1$ on $U$, $\eta=0$ on $X\setminus V$. The term $2/(b-a)$ is then replaced
with a constant $C$ that depends solely on $U,V$. Next, in the proof of quasiminimality, one replaces
$B_{X_\8}({z},[1+\tau]R)$ with $U_\tau$ where $U$ is the support of $\pip$ and $U_\tau=\{y\in X_\8 :\, d_{X_\8}(y, U)<\tau\}$.
In this case, $B_{n}(x_n,[1+\tau]R)$ is replaced with a suitable approximation of $U_\tau$ in $X_n$, ensuring
that this approximating open set is contained within $B_{n}(x,\rho)$ where 
$\rho=2[\text{diam}_{X_\8}(U)+\text{dist}(U,x_\8)]$.
\end{remark} 

\noindent Addresses:\\

\vskip .2cm

\noindent S.E.-B.: Mathematics Department, University of California-Los Angeles, 
520 Portola Plaza, Los Angeles, CA-90025, U.S.A.\\
\\
\noindent E-mail: {\tt syerikss@math.ucla.edu}

\vskip .6cm

\noindent J.T.G.: Department of Mathematics and Statistics, Saint Louis University,
Ritter Hall~307, 220 N. Grand Blvd.,  St. Louis, MO-63103, U.S.A.\\
\\
\noindent E-mail: {\tt jim.gill@slu.edu}

\vskip .6cm

\noindent P.L.: University of Jyvaskyla,
Department of Mathematics and Statistics,
P.O. Box 35, FI-40014 University of Jyvaskyla, Finland.\\
\\
\noindent E-mail: {\tt panu.k.lahti@jyu.fi}

\vskip .6cm

\noindent N.S.: Department of Mathematical Sciences, P.O. Box 210025, 
University of Cincinnati, Cincinnati, OH~45221-0025, U.S.A.\\
\\
\noindent E-mail: {\tt shanmun@uc.edu}

\end{document}